\input ./style/arxiv-general.cfg
\documentclass[aap,MSNbibl,seceqn,dvips]{arximspdf}
\makeatletter
   \@ifpackageloaded{graphicx}{}{\usepackage{graphicx}}
\makeatother
\usepackage{mathbh}

\doi{10.1214/14-AAP1040} %kopijuoti is PTS
\volume{25}
\issue{4}
\pubyear{2015}
\firstpage{1993}
\lastpage{2012}
\docsubty{FLA}

\makeatletter
\newcommand{\rrvert}{\vert}
\newcommand{\llvert}{\vert}
\newcommand{\eqref}[1]{(\ref{#1})}
\newcommand{\R}{\mathbb R}
\newcommand{\Z}{\mathbb Z}
\newcommand{\N}{\mathbb N}
\newcommand{\Prob}{\mathbb P}
\newcommand{\E}{\mathbb E}

\newcommand{\ident}{\mathbh{1}}
\newtheorem{teo}{Theorem}[section]
\newtheorem{lem}[teo]{Lemma}
\newtheorem{cor}[teo]{Corollary}
\newproclaim{rem}[teo]{Remark}

\newproclaim{defn}[teo]{Definition}
\newproclaim{exmp}[teo]{Example}
\newproclaim{assump}[teo]{Assumption}

\makeatother

\begin{document}
\begin{frontmatter}

\title{Branching random walks and multi-type contact-processes on the
percolation\break cluster of ${\mathbb Z}^{\bolds{d}}$}
\runtitle{BRW{s} and multi-type CP{s} on the percolation cluster
of $\mathbb{Z}^d$}

\begin{aug}
\author[A]{\fnms{Daniela}~\snm{Bertacchi}\ead[label=e1]{daniela.bertacchi@unimib.it}}
\and
\author[B]{\fnms{Fabio}~\snm{Zucca}\corref{}\ead[label=e2]{fabio.zucca@polimi.it}}
\runauthor{D.~Bertacchi and F.~Zucca}

\affiliation{Universit\`a di Milano--Bicocca and Politecnico di Milano}

\address[A]{Dipartimento di Matematica e Applicazioni\\
Universit\`a di Milano--Bicocca\\
via Cozzi 53\\
20125 Milano\\
Italy\\
\printead{e1}}

\address[B]{Dipartimento di Matematica\\
Politecnico di Milano\\
Piazza Leonardo da Vinci 32\\
20133 Milano\\
Italy\\
\printead{e2}}
%\author{\fnms{}~\snm{}\corref{}}
%\and
%\author{\fnms{}~\snm{}}
%\runauthor{}
%\affiliation{}
%\dedicated{}
%\address{} %adresu isvedimo komanda gale!
%\address{}
\end{aug}

% HISTORY:
\received{\smonth{11} \syear{2013}}
\revised{\smonth{2} \syear{2014}}
%\accepted{\smonth{} \syear{}}

% ABSTRACT
%
\begin{abstract}
In this paper we prove that, under the assumption of
quasi-transitivity, if a branching random
walk on ${{\mathbb Z}^d}$ survives locally (at arbitrarily large times
there are
individuals alive at the
origin), then so does the same process when restricted to the infinite
percolation
cluster ${{\mathcal C}_\infty}$ of a supercritical Bernoulli
percolation. When no more
than $k$ individuals per site are
allowed, we obtain the $k$-type contact process, which can be derived
from the branching random walk
by killing all particles that are born at a site where already $k$
individuals are present.
We prove that local survival of the branching random
walk on ${{\mathbb Z}^d}$ also implies that for $k$ sufficiently large the
associated $k$-type contact process
survives on ${{\mathcal C}_\infty}$. This implies that the strong
critical parameters
of the branching random walk
on ${{\mathbb Z}^d}$ and on ${{\mathcal C}_\infty}$ coincide and that
their common
value is the
limit of the sequence of strong
critical parameters of the associated $k$-type contact processes.
These results are extended to a family of restrained branching random
walks, that is, branching
random walks where the success of the reproduction trials decreases
with the size of
the population in the target site.
% As a corollary we improve \cite[Theorem~1]{cf:BLZ}.
\end{abstract}

% KEYWORDS
\begin{keyword}[class=AMS]
\kwd[Primary ]{60K35}
%\kwd{60K35}
\kwd[; secondary ]{60K37}
\end{keyword}
\begin{keyword}
\kwd{Branching random walk}
\kwd{contact process}
\kwd{percolation cluster}
\kwd{critical parameters}
\kwd{approximation}
\end{keyword}
%
%\begin{keyword}[class=AMS]
%\kwd[Primary ]{}
%\kwd{}
%\kwd[; secondary ]{}
%\end{keyword}
%\begin{keyword}
%\kwd{}
%\end{keyword}
\end{frontmatter}

%s1 #&#
\section{Introduction}
\label{sec:intro}

% \begin{enumerate}
% \item
% Scrivere in qualche punto che in tempo discreto resta uguale
% \item
% Citare l'ultimo lavoro di Sebastian (referaggio Fabio AAP) e qualche
%lavoro di Popov.
% \end{enumerate}

The branching random walk is a process which serves as a model for a
population living in a spatially structured environment
[the vertices of a graph $(X,\mathcal E(X))$].
Each individual lives in a vertex, breeds and dies at
random times and each offspring is placed (according to some rule) in
one of the
neighboring vertices.
Since for the branching random walk (BRW in short) there is no bound on
the number
of individuals allowed per site, it is natural to consider a
modification of the
process, namely the multitype contact process, where, for some $k\in\N
$, no more
than $k$ particles per site are allowed (if $k=1$ one gets the usual
contact process).
The multitype contact processes are more realistic models. Indeed,
instead of thinking
of the vertices of the graph as small portions of the ecosystem where
individuals may
pile up indefinitely (like in the BRW), here each vertex can host at
most $k$ individuals.
This is, in particular, true for patchy habitats (each vertex
represents a patch of soil)
or in host-symbionts interactions (each vertex represents a host on top
of which symbionts
may live); see, for instance, \mbox{\cite{cf:BL,cf:BBZ,cf:BLZ}}.

The need for more realistic models also brings random environment into
consideration.
BRWs in random environment has been studied by many authors; see, for instance,
\cite{cf:CMP98,cf:GMPV09,cf:DHMP99,cf:MP00,cf:MP03,cf:M08}.
In many cases the random environment is a random choice of the
reproduction law of the
process (in some cases there is no death).
In our case we put the randomness into the underlying graph.
When choosing $(X,\mathcal E(X))$, ${{\mathbb Z}^d}$ is perhaps the
first choice
that comes
to mind, but other graphs are reasonable options. In particular the BRW
and the contact
process have been studied also on trees \cite
{cf:HuLalley,cf:Ligg1,cf:Ligg2,cf:MadrasSchi,cf:Pem,cf:Stacey03} and on
random graphs as Galton--Watson trees \cite{cf:PemStac1}.
% and the percolation cluster of $\Zd$ (popov...).
Although ${{\mathbb Z}^d}$ has clear properties of regularity, which
make it a
nice case to study,
random graphs are believed to serve as a better model for real-life
structures and social
networks. It is therefore of interest to investigate the behavior of
stochastic processes on
random graphs, which possibly retain some regularity properties which
make them treatable.
An example is the small world,
which is the space model in \cite{cf:DurrettJung} and \cite{cf:BerBor},
where each vertex has the same number of neighbors.
The percolation cluster of $\mathbb{Z}^d$ given by a supercritical
Bernoulli percolation, which we denote by $\mathcal C_\infty$,
has no such regularity, but has a ``stochastic'' regularity, and its
geometry, if viewed
at a large scale, does not differ too much from ${{\mathbb Z}^d}$ (e.g.,
it is true that, for large $N$,
in many $N$-boxes of ${{\mathbb Z}^d}\cap\mathcal C_\infty$, there
are open
paths crossing the box in each direction
and these paths connect to crossing paths in neighboring boxes; see
\cite{cf:Grimmett}, Chapter~7).
Indeed $\mathcal C_\infty$ shares many stochastic properties with~${{\mathbb Z}^d}$: the simple random walk
is recurrent in $d=1,2$, transient in $d\ge3$ and the transition
probabilities have the same
space--time asymptotics as those of ${{\mathbb Z}^d}$ (with different constants,
\cite{cf:Barlow}); two walkers collide infinitely many often
in $d=1,2$ and finitely many times in $d\ge3$ (see \cite
{cf:BarPerSou}); the voter model clusters
in $d=1,2$ and coexists in $d\ge3$ (see \cite{cf:BLZ}), just to
mention a few facts.

The aim of this paper is to compare the critical parameters of the BRW
and of the multitype contact
process on the infinite percolation cluster $\mathcal C_\infty$ with
the corresponding ones on ${{\mathbb Z}^d}$
(from now on we tacitly assume that the infinite cluster exists almost
surely, i.e., that the underlying
Bernoulli percolation is supercritical).
In order to define these parameters, let us give a formal definition of
the processes involved.

Let $(X,\mathcal E(X))$ be a graph and $\mu\dvtx X\times X\to
[0,+\infty)$
\textit{adapted to the graph}, that is,
$\mu(x,y)>0$ if and only if $(x,y)\in\mathcal E(X)$. % (in which case
%we write $x\to y$).
% We call the couple $(X,\mu)$ a weighted graph.
We require that there exists $K<+\infty$
such that $\zeta(x):=\sum_{y\in X}\mu(x,y)\le
K$ for all $x\in X$. %(other conditions will be stated in Section~\ref{sec:definitions}).
Given $\lambda>0$, the $\lambda$-\textit{branching random walk}
($\lambda$-BRW or, when $\lambda$ is not relevant, BRW) is the continuous-time
Markov process $\{\eta_t\}_{t\ge0}$, with configuration space $\N^X$,
where each existing particle at $x$ has an exponential lifespan
of parameter 1 and, during its life, breeds at the arrival times
of a Poisson process of parameter $\lambda\zeta(x)$
and then chooses to send its offspring to $y$
with probability $\mu(x,y)/\zeta(x)$. Thus we associate to $\mu$ a
family of BRWs, indexed by $\lambda$.
With a slight abuse of notation, we will say that $(X,\mu)$ is a BRW
[$\mu(x,y)$ represents the rate at
which existing particles at $x$ breed in $y$].
The BRW is called \textit{irreducible} if and only if the underlying
graph is connected.
Clearly, any BRW on ${{\mathbb Z}^d}$ or ${{\mathcal C}_\infty}$ is
irreducible;
we note that in their graph structure we possibly admit loops; that is,
every vertex might be
a neighbor of itself (thus allowing reproduction from a vertex onto itself).
% (note that $(\mu(x,y)/\zeta(x))_{x,y\in X}$ is the transition matrix
%of a random walk on $X$).
If $(Y,\mathcal E(Y))$ is a subgraph of $(X,\mathcal E(X))$, we denote by
$\mu_{|Y}(x,y)$ the map $\mu\cdot\ident_{\mathcal E(Y)}$. The
associated BRW $(Y,\mu_{|Y})$,
% $\mu$ \textit{restricted to} $Y$ is the family of BRWs on $(Y,
%\mathcal E(Y))$,
indexed by $\lambda$, is called the \textit{restriction} of $(X,\mu
)$ to $Y$ and,
to avoid cumbersome notation, we denote it by $(Y,\mu)$.
% associated to $\widetilde\mu$ defined by $\widetilde\mu(x,y)=
%\mu(x,y)$ if $(x,y)\in\mathcal E(Y)$
% and $\widetilde\mu(x,y)=0$ otherwise.

Two critical parameters are associated to the continuous-time BRW: the
weak (or global) survival
critical parameter $\lambda_w$ and the strong (or local) survival one
$\lambda_s$.
They are defined as
%
%e1.1 #&#
\begin{eqnarray}
\label{eq:critical} \lambda_w(x_0)&:=&\inf \bigl\{\lambda>0
\dvtx\Prob^{\delta_{x_0}} (\exists t\dvtx\eta_t=\underline{0} )<1
\bigr\},
\nonumber
\\[-8pt]
\\[-8pt]
\nonumber
\lambda_s(x_0)&:=&\inf \bigl\{\lambda>0\dvtx
\Prob^{\delta_{x_0}} \bigl(\exists\bar t\dvtx\eta_t(x_0)=0,
\forall t\ge\bar t \bigr)<1 \bigr\},
\end{eqnarray}
where $x_0$ is a fixed vertex, $\underline{0}$ is the configuration
with no
particles at all sites and $\Prob^{\delta_{x_0}}$ is the law of the process
which starts with one individual in $x_0$.
Note that these parameters do not depend on the initial state $\ell
\delta_{x_0}$, provided that
$\ell>0$. Moreover, if the BRW is irreducible, then these values
do not depend on the choice of $x_0$ nor on the initial configuration,
provided that this configuration
is nonzero and finite (i.e., it has a strictly positive, finite number
of individuals).
When there is no dependence on $x_0$, we simply write $\lambda_s$ and
$\lambda_w$.
These parameters depend also on $(X,\mu)$: when we need to stress this
dependence,
we write $\lambda_w(x_0,X,\mu)$ and $\lambda_s(x_0,X,\mu)$ [or
simply $\lambda_w(X,\mu)$ and $\lambda_s(X,\mu)$
in the irreducible case].
We refer the reader to Section~\ref{sec:prelim} for how to compute the
explicit value of these
parameters.

Given $(X,\mu)$ and a nonincreasing function $c\dvtx\R^+\to\R^+$, the
\textit{restrained branching random walk}
(briefly, RBRW) $(X,\mu,c)$ is the continuous-time
Markov process $\{\eta_t\}_{t\ge0}$, with configuration space $\N^X$,
where each existing particle at $x$ has an exponential lifespan
of parameter 1 and, during its life, breeds, as the BRW, at rate $c(0)
\zeta(x)$,
then chooses to send its offspring to $y$
with probability $\mu(x,y)/\zeta(x)$, and the reproduction is
successful with
probability $c(\eta(y))/c(0)$. For the RBRW the rate of successful
reproductions
from $x$ to $y$, namely $\mu(x,y)c(\eta(y))$ depends on the
configuration; for a formal introduction to RBRWs, see \cite{cf:BPZ}.

Restrained branching random walks have been introduced in \cite
{cf:BPZ} in order
to provide processes where the natural competition for resources in an
environmental patch
is taken into account (since $c$ is nonincreasing, the more individuals
are present at a vertex,
the more difficult it is for new individuals to be born there). If we
imagine that the vertex
can host at most $N$ individuals, a natural example of $c$ is
represented by the logistic growth
$c_N(i)=\lambda(1-i/N)\ident_{[0,N]}(i)$.
A more general choice where the parameter $N$ represents the strength
of the competition between
individuals (the smaller $N$, the stronger the competition), is given
by fixing a nonincreasing
$\widetilde c$ and letting $c_N(\cdot):=\widetilde c(\cdot/N)$.
The usual BRWs and multitype contact processes can be seen as
particular cases of RBRWs:
if $c\equiv\lambda$, the associated RBRW is the $\lambda$-BRW;
if $c=\lambda\ident_{[0,k-1]}$, we call the corresponding RBRW
$k$-\textit{type contact process}, and we denote it by $\{\eta_t^k\}
_{t\ge0}$.
The critical parameters of the $k$-type contact process are denoted by
$\lambda_s^k$ and $\lambda_w^k$.

The order relations between all these critical values are shown in
Figure~\ref{fig:diagram}; these relations hold
for every $\mu$ adapted to $\Z^d$.

%f1 #&#
\begin{figure}

\includegraphics{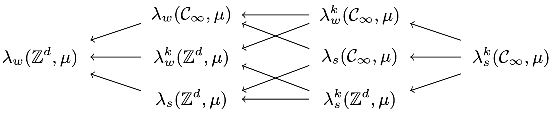}

\caption{Order relation between critical values ($a \to b$ means $a
\ge b$).}
\label{fig:diagram}
\end{figure}

%\begin{figure}[ht]
% \begin{tikzpicture}[scale=0.3]
%\node at (0,0) {$\lambda_w(\Zd,\mu)$};
%\node at (9,-2.5){$\lambda_s(\Zd,\mu)$};
%\node at (9,0){$\lambda_w^k(\Zd,\mu)$};
%\node at (9,+2.5){$\lambda_w(\mathcal{C}_\infty,\mu)$};
%\node at (19,-2.5){$\lambda_s^k(\Zd,\mu)$};
%\node at (19,0){$\lambda_s(\mathcal{C}_\infty,\mu)$};
%\node at (19,2.5){$\lambda_w^k(\mathcal{C}_\infty, \mu)$};
%\node at (28,0){$\lambda_s^k(\mathcal{C}_\infty, \mu)$};
%
%\draw[<-] (3,-1) to (6,-2);
%\draw[<-] (3,0) to (6,0);
%\draw[<-] (3,1) to (6,2);
%
%\draw[->] (25,-1) to (22,-2);
%\draw[->] (25,0) to (22,0);
%\draw[->] (25,1) to (22,2);
%
%\draw[->] (16,-0.5) to (12,-2);
%\draw[->] (16,0.5) to (12,2);
%
%\draw[->] (16,2.5) to (12,2.5);
%\draw[->] (16,2) to (12,0.5);
%
%\draw[->] (16,-2.5) to (12,-2.5);
%\draw[->] (16,-2) to (12,-0.5);
% \end{tikzpicture}
%
%\end{figure}

It has already been proven in \cite{cf:BZ3} that if $\mu$ is
quasi-transitive on $X$
(a property of regularity, see Definition~\ref{def:quasit}), then
$\lambda_s^k(X,\mu)\stackrel{k\to\infty}{\longrightarrow} \lambda
_s(X,\mu)$, and,
if $\mu$ is translation invariant on ${{\mathbb Z}^d}$, then
$\lambda_w^k({{\mathbb Z}^d},\mu)\stackrel{k\to\infty
}{\longrightarrow}
\lambda_w({{\mathbb Z}^d},\mu)$.
Analogous results for discrete-time processes can be found in \cite
{cf:Z1}, and
recently some progress has been made for discrete-time BRWs on Cayley
graphs of finitely generated groups; see \cite{cf:Muller14}.

When considering BRWs and multitype contact processes on $\mathcal
C_\infty$, two natural
questions arise. First, we wonder whether the critical parameters of
the BRW
on $\mathcal C_\infty$ can be deduced from the ones of the BRW on
${{\mathbb Z}^d}
$; second, whether
the parameters of the $k$-type contact process converge to the
corresponding ones of the BRW.
Note that even if the BRW $({{\mathbb Z}^d},\mu)$ has good properties of
regularity, like quasi-transitivity,
its restriction to $\mathcal C_\infty$ has none of these properties,
and the aforementioned
questions are not trivial.

Our main result answers both questions regarding $\lambda_s$: for
quasi-transitive BRWs on ${{\mathbb Z}^d}$
the strong critical parameter coincides with the one on $\mathcal
C_\infty$
(this result was actually already in \cite{cf:BZ3}, Theorem~7.1, but
here we provide a different
proof which can be extended to answer the second question). Moreover
the sequence
of the strong critical parameters of $k$-type contact processes
restricted to $\mathcal C_\infty$
converge to the one of the BRW on ${{\mathbb Z}^d}$.
We note that here we consider only continuous-time processes, but
analogous results hold for
discrete-time BRWs as well.

%th1.1 #&#
\begin{teo}\label{thm:approxs}
Let $({{\mathbb Z}^d},\mu)$ be a quasi-transitive BRW and $C_\infty
\subseteq\Z
^d$ be the infinite cluster of a supercritical Bernoulli
percolation.
% and consider on $\mathcal C_\infty$ the
% restriction $\widetilde\mu$ to $\mathcal C_\infty$.
Then:
\begin{longlist}[(1)]
\item[(1)] $\lambda_s(\mathcal C_\infty,\mu%_{|\mathcal C_\infty}
)=\lambda_s({{\mathbb Z}^d},\mu)$
a.s.~with respect to the realization of ${{\mathcal C}_\infty}$;
\item[(2)] $\lim_{k\to\infty} \lambda_s^k(\mathcal C_\infty,\mu
%%_{|\mathcal C_\infty}
)=\lim_{k\to\infty}
\lambda_s^k({{\mathbb Z}^d},\mu)
% =\lambda_s(\mathcal C_\infty)
=\lambda_s({{\mathbb Z}^d},\mu)
$ a.s.~with respect to the realization of ${{\mathcal C}_\infty}$.
\end{longlist}
\end{teo}

We observe that the equality
$\lim_{k\to\infty}
\lambda_s^k({{\mathbb Z}^d},\mu)
% =\lambda_s(\mathcal C_\infty)
=\lambda_s({{\mathbb Z}^d},\mu)$ has already been proven in \cite{cf:BZ3},
Theorem~5.1.
The result for the weak critical parameter can be obtained when
$\lambda_w({{\mathbb Z}^d},\mu)=\lambda_s({{\mathbb Z}^d},\mu)$,
which is, for instance,
true when
$\mu$ is quasi-transitive and symmetric; see Section~\ref{sec:prelim}.

%th1.2 #&#
\begin{teo}\label{thm:approxw}
Let $({{\mathbb Z}^d},\mu)$ be a quasi-transitive BRW such that
$\lambda_w({{\mathbb Z}^d},\break \mu)=\lambda_s({{\mathbb Z}^d},\mu)$,
and let
$C_\infty\subseteq\Z^d$ be the infinite cluster of a supercritical Bernoulli
percolation.
% Suppose that there exists $\nu$
% such that $\nu(x)\mu(x,y)=\nu(y)\mu(y,x)$ and $\nu(x)/\nu(x)\le c$
%for some $c>0$ and for all $x,y$.
Then, a.s.~with respect to the realization of ${{\mathcal C}_\infty}$,\break
$\lim_{k\to\infty} \lambda_w^k(\mathcal C_\infty,\mu%_{|\mathcal
%C_\infty}
)=\lim_{k\to\infty} \lambda_w^k({{\mathbb Z}^d},\mu)
=\lambda_w(\mathcal C_\infty,\mu%_{|\mathcal C_\infty}
)=\lambda_w({{\mathbb Z}^d},\mu)$.
%=\lambda_s(\Zd,\mu).$
\end{teo}

%% Nota: $\nu$ mappa di reversibilit{\chr"C3}{\chr"A0} non implica che
%il range di $\nu$ sia finito
%% come controesempio basta prendere $\mu$ il drift su $\Z$
% Theorem~\ref{thm:approxw} is a consequence of Theorem~\ref{thm:approxs}
% and of the fact that under the assumption we have that $\lambda_w(\Zd,
%\mu)=\lambda_s(\Zd,\mu)$.

The fact that whenever a quasi-transitive BRW on ${{\mathbb Z}^d}$ is locally
supercritical (i.e.,~$\lambda> \lambda_s$),
so are the $k$-type contact processes
restricted to $\mathcal C_\infty$, whenever $k$ is sufficiently large,
also holds
for families of RBRWs, where %the function
$c_N(\cdot):=c(\cdot/N)$, and $c$ is a given nonnegative
function such that $\lim_{z \to0^+} c(z)=c(0)> \lambda_s({{\mathbb
Z}^d}, \mu)$.

%th1.3 #&#
\begin{teo}\label{thm:RBRW}
Let $({{\mathbb Z}^d},\mu)$ be a quasi-transitive BRW and $C_\infty
\subseteq\Z
^d$ be the infinite cluster of a supercritical Bernoulli
percolation.
Let $c$ be a nonnegative, nonincreasing function such that $\lim_{z
\to0^+} c(z)=c(0)> \lambda_s({{\mathbb Z}^d}, \mu)$,
and let $c_N(\cdot):=c(\cdot/N)$.
Consider the RBRWs $({{\mathbb Z}^d},\mu,c_N)$ and $({{\mathcal
C}_\infty},\mu,c_N)$: they both
survive locally
whenever
$N$ is sufficiently large.
\end{teo}

As an application, we have that \cite{cf:BLZ}, Theorem~1(2), can be
refined; here is the improved statement.

%co1.4 #&#
\begin{cor}\label{cor:BLZ}
Let $\mu(x,x)=\alpha$ and $\mu(x,y)=\beta/2d$ for all $x\in
{{\mathbb Z}^d}$
and $y$ such that $|x-y|=1$,
where $\alpha\ge0$ and $\beta> 0$. Consider the RBRW $({{\mathcal
C}_\infty},\mu,c_N)$ where
$c_N(i)=(1-i/N)\ident_{[0,N]}(i)$. Then:
\begin{longlist}[(1)]
\item[(1)] For all $N > 0$, the process dies out if $\alpha+ \beta
\le1$.
\item[(2)] If $\alpha+ \beta> 1$, then the process survives locally,
provided that $N$ is sufficiently large.
\end{longlist}
\end{cor}

To compare with \cite{cf:BLZ}, Theorem~1, we recall that the
extinction phase, that is,
Corollary~\ref{cor:BLZ}(1),
was already stated as \cite{cf:BLZ}, Theorem~1(1); to ensure survival
when $\alpha+\beta>1$
and $N$ is large, \cite{cf:BLZ}, Theorem~1(2), requires that the
parameter of the underlying
Bernoulli percolation is sufficiently close to 1. This request has now
been proven unnecessary, since
it suffices that the Bernoulli percolation is supercritical.

%s2 #&#
\section{Basic definitions and preliminaries}
\label{sec:prelim}

Explicit characterizations of the critical parameters are possible.
For the strong critical parameter we have $\lambda_s(x)=1/\limsup_{n
\to\infty}\sqrt[n]{\mu^{(n)}(x,x)}$
(see \cite{cf:BZ2}, Theorem 4.1, \cite{cf:BZsl}, Theorem 3.2(1))
where $\mu^{(n)}(x,y)$ are recursively defined by
$\mu^{(n+1)}(x,y)=\break \sum_{w\in X}\mu^{(n)}(x,w)\mu(w,y)$ and $\mu
^{(0)}(x,y)=\delta_{xy}$.
As for $\lambda_w(x)$, it is characterized in terms of solutions of
certain equations
in Banach spaces (see \cite{cf:BZ2}, Theorem~4.2); moreover,
$\lambda_w(x)\ge1/\liminf_{n \to\infty}\sqrt[n]{\sum_{y\in X}\mu
^{(n)}(x,y)}$
\cite{cf:BZ2}, Theorem~4.3, \cite{cf:BZsl}, Theorem~3.2(2).
The last inequality becomes an equality in a certain class of BRWs
which contains
quasi-transitive BRWs (see \cite{cf:BZ2}, Proposition~4.5, \cite{cf:BZsl},
Theorem~3.2(3))
The definition of quasi-transitive BRW is the following.

%de2.1 #&#
\begin{defn}\label{def:quasit}
% \begin{enumerate}
% \item
$(X,\mu)$ is a quasi-transitive BRW (or
$\mu$ is a quasi-transitive BRW on $X$) if and only if there exists a
finite set of vertices
$\{x_1,\ldots,x_r\}$ such
that for every $x\in X$ there exists a bijection $f\dvtx X\to X$
such that $f(x_j)=x$ for some $j$ and $\mu$ is $f$-invariant, that is,
$\mu(w,z)=\mu(f(w),f(z))$ for all $w,z$.
% \item$\mu$ is adapted to $(X,\mathcal E(X))$ if
% \[
% \mu(x,y)>0\ \Leftrightarrow\ x\sim y.
% \]
% \end{enumerate}
\end{defn}

Note that if
% $\mu$ is adapted to $(X,\mathcal E(X))$ and
$f$ is a bijection such that
$\mu$ is $f$-invariant, then $f$ is an automorphism of the graph
$(X,\mathcal{E}(X))$.
In many cases $\lambda_s$ coincides with $\lambda_w$.
For quasi-transitive and symmetric BRWs [i.e., $\mu(x,y)=\mu(y,x)$
for all $x,y$], it is known
that $\lambda_s=\lambda_w$ is equivalent to amenability (\cite{cf:BZsl},
Theorem~3.2,
which is essentially based on \cite{cf:BZ2} and \cite{cf:Stacey03}, Theorem~2.4).
Amenability is a slow growth condition; see~\cite{cf:Stacey03}, Section~1,
for the definition of amenable graph and \cite{cf:BZsl}, Section~2,
where $m_{xy}$ stands for
$\mu(x,y)$, for the definition of amenable BRW.
It is easy to prove that a quasi-transitive BRW is amenable if and only
if the underlying graph
is amenable. Examples of amenable graphs are ${{\mathbb Z}^d}$
along with its subgraphs. Therefore, every quasi-transitive and
symmetric BRW on ${{\mathbb Z}^d}$ or ${{\mathcal C}_\infty}$ has
$\lambda_s=\lambda_w$.

Another %more general
sufficient condition is the following, where symmetry is replaced by
reversibility
[i.e.,~the existence of measure $\nu$ on $X$ such that $\nu(x)\mu
(x,y)=\nu(y)\mu(y,x)$ for all $x,y$].
% The following result gives a useful sufficient condition for the
%absence of pure global
% survival for a continuous-time BRW $(X,K)$
% which is based only on the geometry of the graph $(X,E_K)$ generated
%by the BRW
% (where $E_K:=\{(x,y) \in X^2: k_{xy}>0\}$).
It is a slight generalization of \cite{cf:BZ}, Proposition~2.1 and easily
extends to discrete-time %, nonoriented
BRWs.

%th2.2 #&#
\begin{teo}\label{th:subexponentialgrowth}
Let $(X,\mu)$ be a continuous-time BRW, and let $x_0 \in X$.
Suppose that there exists a measure $\nu$ on $X$ and %$c>0$ such that
$\{c_n\}_{n \in\N}$ such that, for all $n \in\N$
\[
\cases{\nu(y)/\nu(x_0)\le c_n, &\quad $\forall y \in
B(x_0,n)$, \vspace*{2pt}
\cr
\nu(x)\mu(x,y)=\nu(y)\mu(y,x), & \quad$\forall
x,y \in X$,}
\]
where $B(x_0,n)$ is the ball of center $x_0$ and radius $n$.
% w.r.~to the natural distance of the graph $(X,\mathcal E(X))$.
If $\sqrt[n]{c_n} \to1$ and\break
$\sqrt[n]{|B(x_0,n)|} \to1$ as $n \to\infty$, then
$\lambda_s(x_0)=\lambda_w(x_0)$.
% $K_s(x_0,x_0)=K_w(x_0)$ and there is no pure global survival starting
%from $x_0$.
\end{teo}

\begin{pf}%[Proof of Theorem~\ref{th:subexponentialgrowth}]
If we denote by $[x_0]$ the irreducible class of $x_0$, then it is easy
to show that
$\mu^{(n+1)}(x_0,x_0)=\sum_{w\in[x_0]}\mu^{(n)}(x_0,w)\mu(w,x_0)$.
Note that $\nu(x)\*\mu^{(n)}(x,y)=\nu(y)\mu^{(n)}(y,x)$ for all $x,y
\in X$,
$n \in\N$. In particular since $\nu(x_0)>0$, then $\nu$ is strictly
positive on $[x_0]$, and $[x_0]$ is a final class.
Thus, for all $x,y \in[x_0]$
we have $\mu(x,y)>0$ if and only if $\mu(y,x)>0$. This means that the
subgraph $[x_0]$ is nonoriented; hence the natural distance is
well defined and so is the ball $B(x_0,n)$.
Moreover, by the Cauchy--Schwarz inequality, the supermultiplicative
property of $\mu^{(n+1)}(x_0,x_0)$ and Fekete's lemma,
for all $n \in\N\setminus\{0\}$,
\begin{eqnarray*}
\bigl(1/\lambda_s(x_0) \bigr)^{2n} &\ge&
\mu^{(2n)}(x_0,x_0)=\sum
_{y \in[x_0]}\mu^{(n)}(x_0,y)\mu
^{(n)}(y,x_0)
\\
&=& \sum_{y \in B(x_0,n)} \bigl(\mu^{(n)}(x_0,y)
\bigr)^2\frac{\nu(x_0)}{\nu
(y)}\ge \frac{ (\sum_y \mu^{(n)}(x_0,y) )^2}{c_n|B(x_0,n)|}.
\end{eqnarray*}
Hence
\begin{eqnarray*}
\frac{1}{\lambda_s(x_0)} &\le&\frac{1}{\lambda_w(x_0)} \le\liminf_n
\sqrt[n]{\sum_y \mu^{(n)}(x_0,y)}\\
&=&
\liminf_n\sqrt[2n]{\frac{ (\sum_y \mu^{(n)}(x_0,y)
)^2}{c_n|B(x_0,n)|}} \le\frac{1}{\lambda_s(x_0)}.
\end{eqnarray*}
\upqed\end{pf}

The condition $\sqrt[n]{|B(x,n)|} \to1$ is usually called \textit
{subexponential growth}.
Examples of subexponentially growing graphs are euclidean lattices
${{\mathbb Z}^d}$
% The previous result applies, for instance to BRWs on $\Z^d$
or $d$-dimensional combs; see \cite{cf:BZ03} for the definition.
% This result extends easily to discrete-time nonoriented BRWs.
The assumptions of Theorem~\ref{th:subexponentialgrowth} are, for
instance, satisfied, on
subexponentially growing graphs, by
irreducible BRWs with a reversibility measure $\nu$ such that $\nu
(x)\le C$ for all $x\in X$ and
for some $C>0$.

One of the tools in the proof of our results is the fact that if the
BRW survives locally on
a graph $X$; it also survives locally on suitable large subsets
$X_n\subset X$.
This follows from the spatial approximation theorems which
have been proven in a weaker form in \cite{cf:BZ3}, Theorem~3.1, for
continuous-time BRWs
and in a stronger form in \cite{cf:Z1}, Theorem~5.2, for discrete-time
BRWs. The proofs
rely on a lemma on nonnegative matrices and their convergence
parameters, which in its
original form can be found in \cite{cf:Sen}, Theorem~6.8. We restate
here both the
lemma and the approximation theorem. It is worth noting that the
irreducibility assumptions
which were present in \cite{cf:Sen,cf:BZ3,cf:Z1} are here dropped.

Given a nonnegative matrix $M=(m_{xy})_{x,y \in X}$, let
$R(x,y):=1/\break \limsup_{n \to\infty}\sqrt[n]{m^{(n)}(x,y)}$ be the
family, indexed by $x$ and $y$,
of the convergence parameters
[$m^{(n)}(x,y)$ are the entries of the $n$th power matrix $M^n$].
Note that, as recalled earlier in this section, $\lambda_s(x)$
coincides with the convergence
parameter $R(x,x)$ of the matrix $(\mu(x,y))_{x,y \in X}$.
Given a sequence of sets $\{X_n\}_{n \in\N}$ let
$\liminf_{n \to\infty} X_n:= \bigcup_n\bigcap_{k\ge n} X_k$.

%le2.3 #&#
\begin{lem}\label{th:genseneta}
Let $\{X_n\}_{n \in\N}$ be a general sequence of subsets of $X$ such that
$\liminf_{n \to\infty} X_n =X$, and suppose that $M=(m_{xy})_{x,y
\in X}$ is a nonnegative matrix.
Consider a sequence
of nonnegative matrices
$M_n=( {m(n)}_{xy})_{x,y \in X_n}$ such that
$0 \le{m(n)}_{xy} \le m_{xy}$ for all $x,y \in X_n$, $n\in\N$ and
$\lim_{n \to\infty} {m(n)}_{xy} =m_{xy}$ for all $x,y \in X$.
Then for all $x_0 \in X$ we have $_nR(x_0,x_0) \rightarrow R(x_0,x_0)$
[$_nR(x_0,x_0)$ being a convergence parameter of the matrix $M_n$].
\end{lem}

Clearly, if
$M$ is irreducible, then $R(x,y)=R$ does not depend on $x,y \in X$, and
for all $x_0 \in X$ we have $_nR(x_0,x_0) \rightarrow R$.
One can repeat the proof of \cite{cf:Z1}, Theorem~5.2, noting that,
since $R(x_0,x_0)$ depends only on the values of the irreducible class
of $[x_0]$
then $_nR(x_0,x_0) \rightarrow R(x_0,x_0)$ without requiring the whole
matrix $M$ to be irreducible.
The following theorem is the application of Lemma~\ref{th:genseneta}
to the spatial approximation
of continuous-time BRWs [an analogous result holds for discrete-time
BRWs (see \cite{cf:Z1}, Theorem~5.2),
where we can drop the irreducibility assumption].

%th2.4 #&#
\begin{teo}\label{th:spatial}
Let $(X,\mu)$ be a continuous-time BRW, and
let us consider a sequence of continuous-time BRWs $\{(X_n, {\mu_n})\}
_{n\in\N}$
such that\break $\liminf_{n \to\infty} X_n = X$.
Let us suppose that
$\mu_n (x,y) \le\mu(x,y)$ for all $x,y\in X_n$, $n\in\N$
and $\mu_n (x,y) \to\mu(x,y)$ as $n \to\infty$ for all $x,y\in X$.
%\hfill\break
Then, for all $x_0 \in X$, $\lambda_s(x_0,X_n, {\mu_n})\ge\lambda
_s(x_0,X,\mu)$ and
$\lambda_s(x_0,X_n, {\mu_n})
\to\lambda_s(x_0,X,\mu)$ as $n\to\infty$.
\end{teo}

%s3 #&#
\section{Proofs and applications}
\label{sec:proofs}
Before proving our main results, we need to prove some preparatory lemmas.
The first lemma gives a useful expression for the expected value of the
progeny living at time $t$ at
vertex $y$ of a particle which was at $x$ at time 0. Its proof, which
can be found in \cite{cf:BPZ}, Section~3,
is based on the construction of the
process by means of its generator as done in \cite{cf:LiggSpitz}.
The key to the proof is the fact
that the expected value is the solution of a system of differential equations.
Neither Bertacchi, Posta and Zucca \cite{cf:BPZ} nor Liggett and
Spitzer \cite{cf:LiggSpitz} construct the process in our setting;
nevertheless it is not difficult to adapt their construction to our
case; the interested reader can find the details in Remark~\ref{rem:liggett}.

%le3.1 #&#
\begin{lem}\label{lem:expnumber}
For any $\lambda$-BRW on a graph $X$ we have that
\[
\E \bigl(\eta_t(y)|\eta_0=\delta_x
\bigr)=e^{-t}\sum_{n=0}^\infty\mu
^{(n)}(x,y)\frac{(\lambda t)^n}{n!}.
\]
The expected number of descendants of generation $n$ at $y$ at time $t$
(of a particle at $x$ at time $0$) is
\[
e^{-t}\mu^{(n)}(x,y)\frac{(\lambda t)^n}{n!},
\]
and the expected number of descendants of generation $n$ at $\gamma_n$
at time $t$
(of a particle at $\gamma_0$ at time $0$)
along the path $\gamma=(\gamma_0,\ldots,\gamma_n)$ is
\[
e^{-t} \prod_{i=0}^{n-1}\mu(
\gamma_i,\gamma_{i+1})\frac{(\lambda t)^n}{n!}
\]
(in this case only the particles of generation $i+1$ at $\gamma_{i+1}$
which are children of particles of generation $i$ at $\gamma_i$ are
taken into account, for all $i=0,\ldots, n-1$).
\end{lem}

The following lemma shows that whenever a BRW on ${{\mathbb Z}^d}$ survives
locally [i.e.,
$\lambda>\lambda_s({{\mathbb Z}^d},\mu)$], it also survives locally if
restricted to boxes of sufficiently large radius.
We denote by $B(m)=[-m,m]^d\cap{{\mathbb Z}^d}$ the box centered at 0
and by
$x+B(m)$ its translate centered at $x$.

%le3.2 #&#
\begin{lem}\label{lem:superBm}
Let $\mu$ be a BRW on ${{\mathbb Z}^d}$. Then for all $\lambda
>\lambda_s({{\mathbb Z}^d},\mu)$ and for all $x\in{{\mathbb Z}^d}$, there
exists $m(x)\in\N$ such that for all $m \ge m(x)$, $\lambda>\lambda
_s(x+B(m),\mu)$.
Moreover, if $\mu$ is quasi-transitive, then there exists $m_0$ such
that for all $m \ge m_0$,
$\lambda>\sup_{x\in{{\mathbb Z}^d}}\lambda_s(x+B(m),\mu)$.
\end{lem}

\begin{pf}
Let $X=\mathbb{Z}^d$,
$X_n:=(x+B(n))$ and $\mu_n:=\mu\cdot\ident_{X_n \times X_n}$.
By Theorem~\ref{th:spatial} there exists $m$ such that $\lambda
>\lambda_s(X_n,\mu_n)$ for all $n\ge m$.

If $\mu$ is quasi-transitive, there exists a finite set of vertices $\{
x_1,\ldots,x_r\}$ as in
Definition~\ref{def:quasit}.
It is clear that $\lambda_s(A,\mu)=\lambda_s(f(A),\mu)$ for all
$A\subset{{\mathbb Z}^d}$ and for every automorphism $f$
such that $\mu$ is $f$-invariant.
Given $\lambda>\lambda_s({{\mathbb Z}^d},\mu)$, for every $i$ there
exists $m_i$ such that $\lambda>\lambda_s(x_i+B(m_i),\mu)$.
Take $m \ge m_0:=\max_{i=1,\ldots,r}m_i$: by monotonicity
$\lambda_s(x_i+B(m_i),\mu)\ge\lambda_s(x_i+B(m),\mu)$ for all $i$. Thus
$\lambda>\max_{i=1,\ldots,r}\lambda_s(x_i+B(m),\mu)$.
Let $x\in{{\mathbb Z}^d}$ and $f$ as in Definition~\ref{def:quasit}
such that $f(x_j)=x$ for some $j$. Then $\lambda_s(x+B(m),\mu
)=\lambda_s(f(x+B(m)),\mu)=\lambda_s(x_j+B(m),\mu)$ and
$\max_{i=1,\ldots,r}\lambda_s(x_i+B(m),\mu)=\sup_{x\in{{\mathbb Z}^d}
}\lambda_s(x+B(m),\mu)$.
\end{pf}

The following lemma states that for any $\lambda$-BRW on a graph $X$,
with $\lambda>\lambda_s(x)$ %\lambda_s(X,\mu)$,
the expected value of the number of particles in a given site, grows
exponentially in time.

%le3.3 #&#
\begin{lem}\label{lem:expgrowth}
Let $\mu$ be a BRW on a graph $X$, $x\in X$ and $\lambda>\lambda
_s(x)$. %\lambda_s(X,\mu)$,
Let $\{\eta_t\}_{t \ge0}$ be the associated $\lambda$-BRW.
Then there exists $\varepsilon=\varepsilon(x,X)$, $C=C(x,X)$ such that
%
%e3.1 #&#
\begin{equation}
\label{eq:expgrowth} \E \bigl(\eta_t(x)|\eta_0=
\delta_x \bigr)\ge Ce^{\varepsilon t} \qquad\forall t\ge0.
\end{equation}
\end{lem}

\begin{pf}
We follow the idea in the proof of \cite{cf:BZ3}, Lemma~5.1.
We prove~\eqref{eq:expgrowth} for all $t\ge t_1$ for some $t_1$; the assertion
then follows by replacing $C$ with $\min(C,C_1)$, where $C_1=\min_{t
\in[0,t_1]} e^{-\varepsilon t}
\E(\eta_t (x)|\eta_0=\delta_x)$ which exists and it is strictly
positive by continuity
[since $t \mapsto\E(\eta_t(x)|\eta_0=\delta_x)$ is a solution of a
differential equation].

Since $\lambda>\lambda_s(x)$, then $\lambda\sqrt[n]{\mu
^{(n)}(x,x)}>1$ for some $n$.
Therefore there exist $n_0 \ge1$ and $\varepsilon_1>0$ such that
$ \mu^{(n_0)}(x,x)> (\frac{1+\varepsilon_1}{\lambda} )^{n_0}$.
By the supermultiplicativity of the sequence $\mu^{(n)}(x,x)$, for all
$r\in\N$,
\[
\mu^{(n_0r)}(x,x)\ge \biggl(\frac{1+\varepsilon_1}{\lambda} \biggr)^{n_0r}.
\]
Recalling Lemma~\ref{lem:expnumber}, we get
\[
\E \bigl(\eta_t(x)|\eta_0=\delta_x \bigr)
\ge e^{-t}\sum_{r\ge0}\frac
{((1+\varepsilon_1)t)^{n_0r}}{(n_0r)!}.
\]
Let $\bar\lambda:=1+\varepsilon_1$. We can write a lower bound for
the summands in the
previous series:
\[
\frac{(\bar\lambda t)^{n_0r}}{(n_0r)!}\ge\frac{\bar\lambda
t-1}{(\bar\lambda t)^{n_0}-1}\cdot \biggl\{\frac{(\bar\lambda t)^{n_0r}}{(n_0r)!}+
\frac{(\bar\lambda
t)^{n_0r+1}}{(n_0r+1)!}+\cdots+ \frac{(\bar\lambda t)^{n_0(r+1)-1}}{(n_0(r+1)-1)!} \biggr\},
\]
whence, for all $t\ge t_1$ and for some $t_1>0$, the following holds:
\[
\E \bigl(\eta_t(x)|\eta_0=\delta_x \bigr)
\ge e^{-t}\cdot\frac{\bar\lambda
t-1}{(\bar\lambda t)^{n_0}-1}\cdot e^{\bar\lambda t} \ge
\frac{\bar\lambda t-1}{(\bar\lambda t)^{n_0}-1}\cdot e^{\varepsilon_1t}\ge e^{\varepsilon_1t/2}.
\]
\upqed\end{pf}

The following lemma states that, for the BRW on ${{\mathbb Z}^d}$,
given two
vertices $x$ and $y$
(also at a large distance), the expected progeny at $y$ of a particle
at $x$, can be made
arbitrarily large, after a sufficiently large time, even if the process
is restricted
to a large box centered at $x$ plus a fixed path from $x$ to $y$; see
Figure~\ref{fig:xtoy}. The idea of the proof
is that the BRW can stay inside the box until the expected number of
particles at $x$
is large, and then move along the path toward $y$.

%f2 #&#
\begin{figure}

\includegraphics{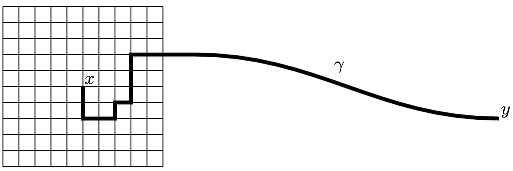}

\caption{The portion of ${{\mathbb Z}^d}$ where we restrict the BRW.}
\label{fig:xtoy}
\end{figure}

%\begin{figure}[ht]
% \begin{tikzpicture}[scale=0.3]
% \node at (0.4,0.4) {$x$};
% \draw[-] (-5,1)--(5,1);
% \draw[-] (-5,2)--(5,2);
% \draw[-] (-5,3)--(5,3);
% \draw[-] (-5,4)--(5,4);
% \draw[-] (-5,5)--(5,5);
% \draw[-] (-5,0)--(5,0);
% \draw[-] (-5,-1)--(5,-1);
% \draw[-] (-5,-2)--(5,-2);
% \draw[-] (-5,-3)--(5,-3);
% \draw[-] (-5,-4)--(5,-4);
% \draw[-] (-5,-5)--(5,-5);
%
% \draw[-] (-5,-5)--(-5,5);
% \draw[-] (-4,-5)--(-4,5);
% \draw[-] (-3,-5)--(-3,5);
% \draw[-] (-2,-5)--(-2,5);
% \draw[-] (-1,-5)--(-1,5);
% \draw[-] (0,-5)--(0,5);
% \draw[-] (5,-5)--(5,5);
% \draw[-] (4,-5)--(4,5);
% \draw[-] (3,-5)--(3,5);
% \draw[-] (2,-5)--(2,5);
% \draw[-] (1,-5)--(1,5);
%
% \draw[line width=2pt](0,0)--(0,-2);
% \draw[line width=2pt](0,-2)--(2,-2);
% \draw[line width=2pt](2,-2)--(2,-1);
% \draw[line width=2pt](2,-1)--(3,-1);
% \draw[line width=2pt](3,-1)--(3,2);
% \draw[line width=2pt](3,2)--(7,2);
%
% \draw[line width=2pt] (7,2) to [out=0, in=180] (26,-2);
% \node at (26.4,-1.6) {$y$};
% \node at (16,1.2) {$\gamma$};
%\end{tikzpicture}
%
%\end{figure}

%le3.4 #&#
\begin{lem}\label{lem:expgrxtoy}
Let $\mu$ be a BRW on ${{\mathbb Z}^d}$, $x\in{{\mathbb Z}^d}$,
$\lambda>\lambda_s({{\mathbb Z}^d},\mu)$.
Fix $M,\delta>0$, and choose $m$ such that
$\lambda>\lambda_s(x+B(m),\mu)$. Then there exists $T=T(x,m,M,\delta
)$ such that
%
%e3.2 #&#
\begin{equation}
\label{eq:xtoy} \E \bigl(\widetilde\eta_t(y)|\widetilde
\eta_0=\delta_x \bigr)\ge1+\delta,
\end{equation}
for all $t\ge T$, $\gamma$ path of length $l\le M$ with $\gamma_0=x$,
$\gamma_l=y$,
where $\{\widetilde\eta_t\}_{t \ge o}$ is the BRW restricted to
$(x+B(m))\cup\gamma$.
Moreover, if $\mu$ is quasi-transitive, we can choose $m$ and $T$
independent of $x$
such that \eqref{eq:xtoy} holds for all $x\in{{\mathbb Z}^d}$.
\end{lem}

\begin{pf}
Fix $t_2>0$. We use the Markov property of the BRW (and the
superimposition with respect to the initial
condition) and apply Lemma~\ref{lem:expnumber}
\begin{eqnarray*}
\E \bigl(\widetilde\eta_{t_1+t_2}(y)|\widetilde\eta_0=
\delta_x \bigr)&\ge& \E \bigl(\widetilde\eta_{t_1}(x)|
\widetilde \eta_0=\delta_x \bigr)\cdot e^{-t_2}
\prod_{i=0}^{l-1}\mu(\gamma_i,
\gamma_{i+1})\frac{(\lambda
t_2)^l}{l!}
\\
& \ge& \E \bigl(\widetilde\eta_{t_1}(x)|\widetilde\eta_0=
\delta_x \bigr)\cdot e^{-t_2}\frac{(\lambda t_2\alpha)^l}{l!}\ge\E \bigl(
\widetilde\eta_{t_1}(x) |\widetilde\eta_0=
\delta_x \bigr)\cdot\widetilde\varepsilon,
\end{eqnarray*}
where $0<\alpha=\alpha(x,M)=\min \{\mu(\gamma^\prime
_i,\gamma^\prime_{i+1})\dvtx i=0,\ldots, l^\prime-1, \gamma
^\prime$ path of length $ 
l^\prime\le M, \gamma^\prime_0=x \}$ and $0<\widetilde
\varepsilon=\widetilde\varepsilon(x,t_2,m,M)=
\min \{e^{-t_2}(\lambda t_2\alpha)^l/l!\dvtx l\le M \}$.\break
Since $\widetilde\eta$ restricted to $x+B(m)$ survives locally, by
Lemma~\ref{lem:expgrowth},
\[
\E \bigl(\widetilde\eta_{t_1+t_2}(y)|\widetilde\eta_0=
\delta_x \bigr)\ge Ce^{\varepsilon t_1}\cdot\widetilde\varepsilon \ge1+
\delta,
\]
for all sufficiently large $t_1$ depending on $x$, $m$, $M$ and $\delta
$. Fix $t_1$ and define $T(x, m, M, \delta):= t_1+t_2$.

If $\mu$ is quasi-transitive, take $\{x_1,\ldots,x_r\}$ and $m_i$ as
in the proof of Lemma \ref{lem:superBm}.
Take $m:=\max_{i=1,\ldots,r}m_i$
and $T=\max_{i=1,\ldots,r} T(x_i,m,M,\delta)$, and the proof is complete.
\end{pf}

In the next lemma we prove that given $x$, $y$ and $y^\prime$,
if we start the process with $l$ particles at $x$, after a sufficiently
large time,
with arbitrarily large probability, we will have $l$ particles both at
$y$ and at
$y^\prime$, even if we restrict the process to a large box centered at
$x$ plus a fixed path from $x$ to $y$
and a fixed path from $x$ to $y^\prime$; see Figure~\ref{fig:xtoyy}.
The proof relies on
Lemma~\ref{lem:expgrxtoy} and the central limit theorem.

%f3 #&#
\begin{figure}

\includegraphics{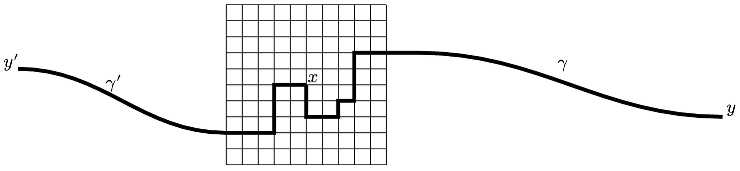}

\caption{From $\ell$ individuals at $x$ to $\ell$ individuals
at $y$ and $y^\prime$.}
\label{fig:xtoyy}
\end{figure}

%\begin{figure}[ht]
% \begin{tikzpicture}[scale=0.3]
% \node at (0.4,0.4) {$x$};
% \draw[-] (-5,1)--(5,1);
% \draw[-] (-5,2)--(5,2);
% \draw[-] (-5,3)--(5,3);
% \draw[-] (-5,4)--(5,4);
% \draw[-] (-5,5)--(5,5);
% \draw[-] (-5,0)--(5,0);
% \draw[-] (-5,-1)--(5,-1);
% \draw[-] (-5,-2)--(5,-2);
% \draw[-] (-5,-3)--(5,-3);
% \draw[-] (-5,-4)--(5,-4);
% \draw[-] (-5,-5)--(5,-5);
%
% \draw[-] (-5,-5)--(-5,5);
% \draw[-] (-4,-5)--(-4,5);
% \draw[-] (-3,-5)--(-3,5);
% \draw[-] (-2,-5)--(-2,5);
% \draw[-] (-1,-5)--(-1,5);
% \draw[-] (0,-5)--(0,5);
% \draw[-] (5,-5)--(5,5);
% \draw[-] (4,-5)--(4,5);
% \draw[-] (3,-5)--(3,5);
% \draw[-] (2,-5)--(2,5);
% \draw[-] (1,-5)--(1,5);
%
% \draw[line width=2pt](0,0)--(0,-2);
% \draw[line width=2pt](0,-2)--(2,-2);
% \draw[line width=2pt](2,-2)--(2,-1);
% \draw[line width=2pt](2,-1)--(3,-1);
% \draw[line width=2pt](3,-1)--(3,2);
% \draw[line width=2pt](3,2)--(7,2);
%
% \draw[line width=2pt] (7,2) to [out=0, in=180] (26,-2);
% \node at (26.4,-1.6) {$y$};
% \node at (16,1.2) {$\gamma$};
%
% \draw[line width=2pt](0,0)--(-2,0);
% \draw[line width=2pt](-2,0)--(-2,-3);
% \draw[line width=2pt](-2,-3)--(-5,-3);
%
% \draw[line width=2pt] (-5,-3) to [out=180, in=0] (-18,1);
% \node at (-12,0.1) {$\gamma^\prime$};
% \node at (-18.4,1.4) {$y^\prime$};
%\end{tikzpicture}
%
%\end{figure}

%le3.5 #&#
\begin{lem}\label{lem:probk}
Let $\mu$ be a BRW on ${{\mathbb Z}^d}$, and let $x$, $\lambda$ and
$m$ as in
Lemma~\ref{lem:expgrxtoy}.
Fix $M,\varepsilon>0$. Then choosing $T=T(x,m,M,1)$ as in Lemma~\ref
{lem:expgrxtoy}, for all $t\ge T$ there exists
$\ell(\varepsilon,x,m,M,t) \in\mathbb{N}$ and
%
%e3.3 #&#
\begin{equation}
\label{eq:probk} \Prob \bigl(\widetilde\eta_t(y)\ge\ell,\widetilde
\eta_t \bigl(y^\prime \bigr)\ge \ell|\widetilde
\eta_0(x)=\ell \bigr)> 1-\varepsilon,
\end{equation}
for all $\ell\ge\ell(\varepsilon,x,t)$, $l,l^\prime\le M$, $\gamma
$, $\gamma^\prime$ paths of length $l$ and $l^\prime$
from $x$ to $y$ and to~$y^\prime$, respectively,
where $\widetilde\eta_t$ is the BRW restricted to $(x+B(m))\cup
\gamma\cup\gamma^\prime$.
Moreover, if $\mu$ is quasi-transitive, we can choose $m$ and $T$
independent of $x$ and $\ell(\varepsilon,m,M)$
such that \eqref{eq:probk} holds for all $x\in{{\mathbb Z}^d}$ when $t=T$.
\end{lem}

\begin{pf}
By monotonicity it suffices to prove the result with the event
$(\widetilde\eta_0=\ell\delta_x)$
in place of $(\widetilde\eta_0(x)=\ell)$.

Let $X=(x+B(m))\cup\gamma\cup\gamma^\prime$.
Let us denote by $\{\xi_t\}_{t \ge0}$ the BRW, restricted to $X$,
starting from $\xi_0=\delta_x$.
By Lemma~\ref{lem:expgrxtoy}, there exists $T$
such that $\E(\xi_t(z)|\xi_0=\delta_x)>2$ for all $t\ge T$,
$z=y,y^\prime$. %and for some positive $\delta$.
A realization of our process is $\widetilde\eta_t=\sum_{j=1}^\ell
\xi_{t,j}$ where
$\{\xi_{t,j}(y)\}_{j\in\N}$
is an i.i.d. family of copies of $\{\xi_t\}_{t \ge0}$.
Fix $z\in\{y,y^\prime\}$.
% with $\E(\xi_{t,j}(y))=\E^{\delta_x}(\xi_t(y))$ and $\mathrm{Var}(
%\xi_{t,j}(y))=:\sigma^2_{t,y}$.
Since $\xi_{t,j}$ is stochastically dominated by a continuous time
branching process with birth rate
$\lambda\sup_w\sum_{v}\mu(w,v)<+\infty$,
it is clear that $\mathrm{Var}(\xi_{t,j}(z))=:\sigma^2_{t,z}<+\infty
$ (note that the variance depends
on $x$).
Thus by the central limit theorem, if $\ell$ is sufficiently large,
\begin{eqnarray*}
&&\frac\varepsilon4\ge \Biggl\llvert \Prob \Biggl(\sum
_{j=1}^\ell\xi _{t,j}(z)\ge s \bigg|
\xi_{0,j}= \delta_x, \forall j=1,\ldots, \ell \Biggr) \\
&&\hspace*{63pt}{}-1+\phi
\biggl(\frac{s-\ell\E(\xi_t(z)|\xi_0= \delta_x)}{\sqrt
\ell\sigma_{t,z}} \biggr) \Biggr\rrvert
\end{eqnarray*}
uniformly with respect to $s\in\R$, where $\phi$ is the cumulative
distribution function of the standard normal.
Whence there exists $\ell(\varepsilon,x,m,M,z,t)$ such that, for all
$\ell\ge\ell(\varepsilon,x,m,M,z,t)$,
\[
\Prob \bigl(\widetilde\eta_{t}(z)\ge\ell |\widetilde
\eta_0= \ell\delta_x \bigr)\ge 1-\phi \biggl(\sqrt\ell
\frac{1-\E(\widetilde\eta_t(y)|\widetilde
\eta_0={\delta_x})}{\sigma_{t,z}} \biggr)- \frac\varepsilon4\ge 1-\frac\varepsilon2,
\]
since $\sqrt\ell(1-\E(\widetilde\eta_t(z)|\widetilde\eta
_0={\delta_x})/\sigma_{t,z}\to-\infty$ as $\ell\to+\infty$.
Take $\ell(\varepsilon,x,m,M,t):=%\max(
\ell(\varepsilon,x,m,M,y,t)%,
\vee\ell(\varepsilon,x,m,M,y^\prime,t)
%)
$.
Hence \eqref{eq:probk} follows.

If $\mu$ is quasi-transitive, take $\{x_i\}_{i=1}^r$ and $\{m_i\}_{i=1}^r$
as in the proof of Lemma~\ref{lem:expgrxtoy}. It suffices to choose
$m:=\max_{i=1,\ldots,r}m_i$
and $T=\max_{i=1,\ldots,r} T(x_i,\break m,M)$.
\end{pf}

We say that a subset $A$ of ${{\mathbb Z}^d}$ is contained in
$\mathcal C_\infty$ if all the vertices are connected to $\mathcal
C_\infty$ and
all the edges $(x,y)$, with $x,y\in A$, are open.
The following is a lemma on the geometry of $\mathcal C_\infty$ which
states that $\mathcal C_\infty$
contains a biinfinite open path where one can find large boxes at
bounded distance from each other.

%le3.6 #&#
\begin{lem}\label{lem:geomcluster}
Let us consider a supercritical Bernoulli percolation on ${{\mathbb Z}^d}$.
For every $m\in\N$ there exists $M=M(m)>0$ such that, a.s.~with
respect to the percolation measure,
the infinite percolation cluster $\mathcal C_\infty$
contains a pairwise disjoint family $\{B_j\}_{j=-\infty}^{+\infty}$
with the following properties:
\begin{longlist}[(1)]
\item[(1)] there exists $\{x_j\}_{j=-\infty}^{+\infty}$, $x_j\in
{{\mathbb Z}^d}
$ for all $j$, and
$B_j=x_j+B(m)$ for all $j$;
\item[(2)] there is a family of open paths $\{\pi_j\}_{j=-\infty
}^{+\infty}$ such that
$x_j\stackrel{\pi_j}{\longleftrightarrow}x_{j+1}$, and $|\pi_j|\le
M$ for all $j$.
\end{longlist}
\end{lem}

\begin{pf}
For every $N \in\mathbb{N}\setminus\{0\}$, we define the
$N$-partition of $\mathbb{Z}^d$ as the collection
$\{2N x + B(N)\dvtx x \in\mathbb{Z}^d\}$.

We use \cite{cf:Pisz}, Proposition~4.1, which holds also for $d=2$
according to \cite{cf:CM2004}, Proposition~11.
In order to achieve in \cite{cf:CM2004}, Proposition~11, the same
generality of \cite{cf:Pisz}, Proposition~4.1,
one has to take into account also a general family of events $\{
V_\Gamma\}_{\Gamma}$
(indexed on the boxes of the collection of the $N$-partitions as $N \in
\mathbb{N}\setminus\{0\}$)
satisfying
equation (4.4) of \cite{cf:Pisz}. This can be easily done by noting
that the inequality (4.25) of \cite{cf:Pisz}
still holds in the case $d=2$.
From now on, when we refer to \cite{cf:Pisz}, Proposition~4.1, we mean
this ``enhanced'' version which holds
for $d \ge2$.

We define $V_{\Gamma}:={}$``there exists a seed $x_\Gamma+B(N^{1/2})
\subseteq\Gamma$'' where
by $\mathit{seed}$ we mean a box with no close edges in the
percolation process
(to avoid a cumbersome notation, we omit the integer part symbol
$\llcorner\cdot\lrcorner$ in the side length).
Note that $V_{\Gamma}$ is measurable with respect to the $\sigma
$-algebra of the percolation process restricted to~$\Gamma$,
thus independent from the rest of the process.
Given a box $\Gamma$ of side length $N$, by partitioning it into
disjoint boxes of side length
$N^{1/2}$, we obtain the following upper estimate
$\bolds{\Phi}[(V_{\Gamma}^c)] \le
(1-p)^{(N/N^{1/2})^d}=(1-p)^{N^{d/2}}\to0$ as $n \to\infty$,
where $\bolds{\Phi}$ is the law of the Bernoulli percolation on
$\mathbb{Z}^d$ with parameter $p$.
This implies that $\{V_\Gamma\}_{\Gamma}$ satisfies equation (4.4) of
\cite{cf:Pisz}.

We know from \cite{cf:Pisz}, Proposition~4.1, that, for any fixed
supercritical Bernoulli percolation on ${{\mathbb Z}^d}$
(``microscopic'' percolation in this context), for every sufficiently
large $N$ the renormalized percolation
(``macroscopic'' percolation from now on) stochastically dominates a
Bernoulli site percolation of arbitrarily large parameter.
Let us describe briefly, how the macroscopic percolation is constructed
from the microscopic one.
For every $k=\pm1,\ldots, \pm d$ we define the $k$th face of
the box $B(N)$ as the
set $\{y \in{{\mathbb Z}^d}\dvtx y(|k|)=\operatorname{sgn}(k)N\}$, that
is, the face
in the $k$th direction.
Roughly speaking, in the renormalized macroscopic process %described in
%\cite{cf:CM2004, cf:Pisz}
a box
$\Gamma:=2N x + B(N)$ ($x \in{{\mathbb Z}^d}$) of the $N$-partition
is \textit
{occupied} if and only if:
\begin{longlist}[(1)]
\item[(1)] there exists a unique \textit{crossing cluster}, that is,
a set of open edges containing open paths connecting
any two opposite faces of the boxes,
\item[(2)] any open path $\gamma$ such that $\operatorname{diam}(\gamma)
\ge N^{1/2}/10$ is connected to
the crossing cluster,
\item[(3)] for every $k=\pm1,\ldots, \pm d$,
if $D_k$ is a translation of the box $B(N/4)$ centered at the middle
point of the $k$th face of the box
$2N x + B(N)$, then there exists a path connecting the $k$ face and the
$-k$ face of $D_k$,
\item[(4)] $V_{\Gamma}$ holds.
\end{longlist}
If $\Gamma$ and $\Gamma^\prime:=2N x^\prime+ B(N)$ are occupied,
where $x^\prime(i)-x(i)=\delta_{i,k}$ (i.e.,
$\Gamma^\prime$ is adjacent
to $\Gamma$ in the $k$th direction), then the crossing clusters of
these two boxes are connected by (2), (3) and
by noting that $D_k=D_{-k}^\prime$ [where $D_k$ and $D_{-k}^\prime$
are the boxes described in (3) related to $\Gamma$
and $\Gamma^\prime$, resp.].

We consider $N > m^{2} \vee23$ ($N \ge24$ is required in \cite
{cf:CM2004,cf:Pisz}).
Thus the seed $x_\Gamma+B(N^{1/2})\subseteq\Gamma
%x+B_{\Gamma}
$, when it exists,
it contains an open path of diameter $2dN^{1/2}>N^{1/2}/10$, and hence
it is connected to the crossing cluster in $\Gamma$
by construction of the renormalized process; see \cite{cf:Pisz},
Section~4.2 or \cite{cf:CM2004}, Section~5. Moreover
it contains a translated box $x_\Gamma+B(m)$. By \cite{cf:Pisz}, Proposition~4.1, given a supercritical Bernoulli percolation
on ${{\mathbb Z}^d}$, %$p > p_{\mathbb{Z}^d}$,
for all sufficiently large $N$,
there exists an infinite open cluster of boxes in the ``macroscopic''
renormalized graph; see \cite{cf:Pisz} for details on the definition of
occupied box.
This implies the existence of an infinite
cluster (in the original microscopic percolation) which contains a seed
no smaller than the box $B(N^{1/2})$ in each
occupied box of the macroscopic cluster; see Figure~\ref{fig:seeds}
where the grayed boxes are occupied.

%f4 #&#
\begin{figure}

\includegraphics{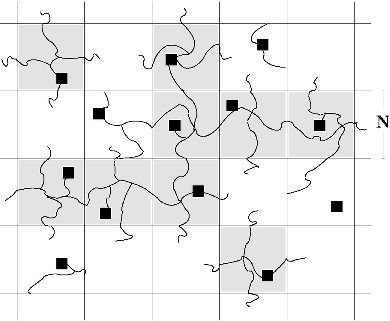}

\caption{The occupied boxes in the renormalized percolation.}
\label{fig:seeds}
\end{figure}

By uniqueness, this infinite microscopic cluster coincides with
$\mathcal{C}_\infty$.
Clearly, by construction, the centers of the seeds in two adjacent
occupied ``macroscopic'' boxes are connected (in
$\mathcal{C}_\infty$) by a path contained into these two boxes;
clearly, the length of such a path is bounded from above by $M:=2N^{d}$.
Since the percolation cluster in the renormalized ``macroscopic''
process contains a bi-infinite self-avoiding path of open boxes,
the proof is complete.
\end{pf}

\begin{pf*}{Proof of Theorem~\ref{thm:approxs}}
Even though (1) follows easily from (2) and the diagram in Figure~\ref
{fig:diagram},
we prove it separately in order to introduce the key idea, which will
be used later
to prove (2), in a simpler case.
(1) Since ${{\mathcal C}_\infty}$ is a subgraph of ${{\mathbb Z}^d}$,
we have that
$\lambda_s(\mathcal C_\infty,\mu%_{|\mathcal C_\infty}
)\ge\lambda_s({{\mathbb Z}^d},\mu)$ (remember that these critical
values do not
depend on the finite, nonzero initial condition).
Take $\lambda>\lambda_s({{\mathbb Z}^d},\mu)$: our goal is to prove that
$\lambda>\lambda_s(\mathcal C_\infty,\mu%_{|\mathcal C_\infty}
)$.
By Lemma~\ref{lem:superBm} we know that there exists (a smallest) $m$
such that
$\lambda>\lambda_s(x+B(m),\mu)$ for all $x \in{{\mathbb Z}^d}$. Let
$M$, $\{
x_j\}_{j=-\infty}^{+\infty}$,
$\{\pi_j\}_{j=-\infty}^{+\infty}$ as in Lemma~\ref{lem:geomcluster}.
By Lemma~\ref{lem:probk} and by monotonicity,
for all $\varepsilon>0$ there exist $T$ and $\ell$ such that
\[
\Prob \bigl(\widetilde\eta_T(x_{j-1})\ge\ell,\widetilde
\eta _T(x_{j+1})\ge \ell|\widetilde\eta_0(x_j)=
\ell \bigr)> 1-\varepsilon,
\]
where $\{\widetilde\eta_t\}_{t\ge0}$ is the BRW (starting from the
initial condition $\ell\delta_{x_0}$) restricted to
$\mathcal A=\bigcup_{j=-\infty}^\infty(x_j+B(m))\cup\pi_j$ (which,
by Lemma~\ref{lem:geomcluster}, is a subset
of ${{\mathcal C}_\infty}$ which exists a.s.~whenever the cluster is
infinite). We
recall that the critical parameters
of the BRW are independent of $\ell>0$.

We construct a process $\{{\xi}_t\}_{t\ge0}$ on $\mathcal A$, by
iteration of independent copies of
$\{\widetilde\eta_t\}_{t\ge0}$ on time intervals $[nT,(n+1)T)$, and
we associate it with a percolation
process $\varrho$ on $\Z\times\vec{\N}$ ($\Z$ representing space
and $\vec\N$ representing time),
where $\vec{\N}$ is the oriented graph on $\N$ where all edges are
of the type $(n,n+1)$.
% If we repeat this construction using all particles at times $nT$
%using copes of $\widetilde\eta$ we obtain
% by the Markov property a process with the same law as the BRW
%\widetilde\eta.
We index the family of copies needed as $\{\widetilde\eta_{(i,j)}\}
_{i\in\Z,j\in\N}$ and use
$\widetilde\eta_{(i,j),t}$ when also the dependence on time has to be
stressed; moreover
$\widetilde\eta_{(i,j),0}=\ell\delta_{x_i}$ for all $i,j$.
The construction will be made in such a way that $\widetilde\eta_t$
stochastically dominates
$\xi_t$ for all $t\ge0$ and, whenever in the percolation process
$\varrho$ we have that
$(0,0)\stackrel{\varrho}{\rightarrow}(j,n)$, then $\xi_{nT}(x_j)\ge
\ell$.
% Approximating p 469

Let us begin our iterative construction with its first step.
Start $\{\widetilde\eta_{(0,0),t}\}_{t\ge0}$,
%from $\ell\delta_{x_0}$
and let $\xi_t=\widetilde\eta_{(0,0),t}$
for $t\in[0,T]$; thus $\xi_0=\widetilde\eta_{(0,0),0}=\ell\delta
_{x_0}$. In the percolation process, the edge $(0,0)\stackrel{\varrho
}{\rightarrow}(j,1)$, $j=\pm1$, is open
if $\widetilde\eta_{(0,0),T}(x_j)\ge\ell$.
Now suppose that we constructed $\{{\xi}_t\}_{t\ge0}$ for $t\in
[0,nT]$; to construct it for $t\in(nT,(n+1)T]$,
% we start $\{\widetilde\eta_{(i,n),t}\}_{t\ge0}$ from time 0, with
%exactly $\ell$ individuals
% in all $x_j$ such that $\xi_{nT}(x_j)\ge\ell$, and
we put $\xi_t=\sum_{h\in[-n,n]\dvtx\xi_{nT}(x_h)\ge\ell
}\widetilde
\eta_{(h,n),t-nT}$ for all
$t\in(nT,(n+1)T]$.
In the percolation~$\varrho$, for all $(i,n)$ such that there is an
open path
$(0,0)\stackrel{\varrho}{\rightarrow}(i,n)$, we connect
$(i,n)\stackrel{\varrho}{\rightarrow}(j,n+1)$, $j=i\pm1$,
if $\widetilde\eta_{(i,n),T}(x_j)\ge\ell$. %$\xi_{(n+1)T}(x_j)\ge
%\ell$.

%f5 #&#
\begin{figure}

\includegraphics{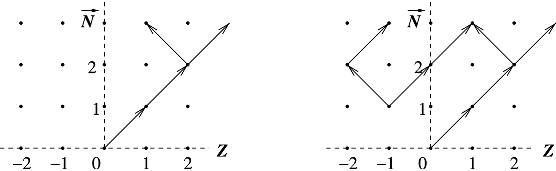}

\caption{A realization of the cluster in the percolation $\varrho$
(left) and $\varrho_2$ (right).}
\label{fig:percolation}
\end{figure}

In order to show that, by choosing $\ell$ sufficiently large, %we have
%that
with positive probability there is an open path in the percolation
$\varrho$, from $(0,0)$
to $(0,n)$ for infinitely many~$n$ (which means that at arbitrarily
large times there are at least
$\ell$ individuals at $x_0$ in the original process), we need a
comparison with a
one-dependent oriented percolation $\varrho_2$ on $\Z\times\vec\N$.
This new percolation $\varrho_2$ is obtained by ``enlarging'' $\varrho
$ in the following way:
for all $(i,n)\in\Z\times\vec\N$, we connect
$(i,n)\stackrel{\varrho_2}{\rightarrow}(j,n+1)$, $j=i\pm1$, if
$\widetilde\eta_{(i,n),T}(x_j)\ge\ell$.
Note that $\varrho$ differs from $\varrho_2$ simply in the fact that
in $\varrho$ the opening procedure
takes place only from sites already connected to $(0,0)$ (see Figure~\ref{fig:percolation}).
By induction on $n$, this coupled construction implies that there
exists a $\varrho_2$-open path
from $(0,0)$ to $(i,n)$ if and only if there exists a $\varrho$-open
path from $(0,0)$ to $(i,n)$.
By Lemma~\ref{lem:probk}, for all $\varepsilon>0$, by choosing $\ell
$ sufficiently large, we have that
for $\varrho_2$ the probability of opening all edges from $(i,n)$ is
at least $1-\varepsilon$.
% Since, if $\varepsilon$ is sufficiently small, the percolation cluster
% of $\varrho_2$ contains infinitely many vertices $(0,n)$, then we
%have local survival for
% $\{\widetilde\eta_t\}_{t\ge0}$.
Let us choose $\varepsilon$ such that the one-dependent percolation
$\varrho_2$ dominates a supercritical
independent (oriented) Bernoulli percolation.
%The sites \{i \dvtx(i,j) connected to (0,0)\} (depending on $j$) is
%a discrete-time contact process on $\mathcal{Z}$
According to Lemma~\ref{lem:recurrence}, the infinite Bernoulli
percolation cluster
in the cone $\{(i,j) \dvtx j \ge|i|\}$
contains infinitely many sites of type $(0,n)$ almost surely.
Hence, by coupling, there is a positive probability that the one-dependent
infinite percolation cluster contains infinitely many sites of type
$(0,n)$ as well.

The first claim follows since the $\lambda$-BRW on ${{\mathcal
C}_\infty}$ (starting
with $\ell$ particles at ${x_0}$)
stochastically dominates
$\{\widetilde\eta_t\}_{t\ge0}$, which in turn dominates $\{\xi_t\}
_{t\ge0}$, and by comparison
with $\varrho_2$ we know that $\xi_{nT}(x_0)\ge\ell$ for infinitely
many $n \in\N$.

(2) Let us now consider the $k$-type contact process $\{\eta^k_t\}
_{t\ge0}$.
Take $\lambda>\lambda_s({{\mathbb Z}^d},\mu)$, $m$ as in the
previous step,
%%given by Lemma~\ref{lem:superBm},
and $\mathcal A$
(along with $\{x_j\}_{j=-\infty}^{+\infty}$ and $\{\pi_j\}
_{j=-\infty}^{+\infty}$)
given by Lemma~\ref{lem:geomcluster} as before.
Consider the restriction $\{\widetilde\eta^k_t\}_{t\ge0}$ of the
$k$-type contact process to $\mathcal A$.
Let us begin by proving that
$\lambda>\lambda_s^k(\mathcal C_\infty,\mu)$ for all $k$
sufficiently large.
To this aim it is enough to prove that for the above fixed $\lambda$,
$\{\widetilde\eta^k_t\}_{t\ge0}$
survives locally for all $k$ sufficiently large.

Fix $\varepsilon>0$, and let $T$ and $\ell$ be given by Lemma~\ref
{lem:probk}, such that
\[
\Prob \bigl(\widetilde\eta_T(y)\ge\ell,\widetilde
\eta_T \bigl(y^\prime \bigr)\ge \ell|\widetilde
\eta_0=\ell \delta_x \bigr)> 1-\varepsilon.
\]
Let $N^x_T$ be the total progeny up to time $T$ (including the initial
particles), in
the BRW $(\mathcal A,\mu)$, starting from $\ell$ individuals at site $x$.
Define $N_T$ as the total number of individuals ever born (including
the initial particles), up to time $T$, in a branching process
with rate $\lambda K$, starting with $\ell$ individuals at time 0:
$N_T$ stochastically dominates $N_T^x$
for all $x\in\mathcal A$.
%Consider $\Prob(\widetilde\eta_T(y)\ge\ell,\widetilde\eta_T(y^
%\prime)\ge\ell|\widetilde\eta_0=\ell\delta_x,N_T^x\le n)$:
%it is equal to
We have
\begin{eqnarray*}
% &\frac{\Prob(\widetilde\eta_T(y) \ge\ell,\widetilde\eta_T(y^\prime)
%\ge\ell,N_T^x\le n|\widetilde\eta_0=\ell\delta_x)}
% {\Prob(N_T^x\le n|\widetilde\eta_0=\ell\delta_x)} \\
% &\phantom{\Prob(\widetilde\eta_T(y) \ge\ell,\widetilde\eta_T(y^
%\prime)}
%\ge
&&\Prob \bigl(
\widetilde\eta_T(y)\ge\ell,\widetilde\eta_T
\bigl(y^\prime \bigr)\ge \ell,N_T^x\le n|
\widetilde\eta_0=\ell\delta_x \bigr)
\\
% &\phantom{\Prob(\widetilde\eta_T(y) \ge\ell,\widetilde\eta_T(y^
%\prime)}
&&\qquad \ge\Prob \bigl(\widetilde\eta_T(y)\ge\ell,
\widetilde\eta_T \bigl(y^\prime \bigr)\ge\ell|\widetilde
\eta_0=\ell\delta_x \bigr) +\Prob \bigl(N_T^x
\le n|\widetilde\eta_0=\ell\delta_x \bigr)-1
\\
% &\phantom{\Prob(\widetilde\eta_T(y) \ge\ell,\widetilde\eta_T(y^
%\prime)}
&&\qquad\ge\Prob \bigl(\widetilde\eta_T(y)\ge\ell,
\widetilde\eta_T \bigl(y^\prime \bigr)\ge\ell|\widetilde
\eta_0=\ell\delta_x \bigr) +\Prob(N_T\le
n)-1 > 1-2\varepsilon,
\end{eqnarray*}
for all $n\ge\bar n$ where $\bar n$ satisfies $\Prob(N_T\le\bar n)>
1-\varepsilon$
($\bar n$ is independent of $x$).

Define an auxiliary process $\{\bar\eta_t\}_{t\in[0,T]}$ obtained
from $\{\widetilde\eta_t\}_{t\ge0}$
by killing all newborns after that the total progeny has reached size
$\bar n$.
This implies that, in the process $\{\bar\eta_t\}_{t\in[0,T]}$, the
progeny does not reach sites at distance
larger than $\bar n$ from the $\ell$ ancestors, nor it goes beyond the
$\bar n$th generation.
In particular, when started from $\ell\delta_x$, the processes $\{
\bar\eta_t\}_{t\in[0,T]}$ and
$\{\widetilde\eta_t\}_{t\ge0}$ coincide, up to time $T$, on the
event $(N_T\le\bar n)$.
Thus
\begin{eqnarray*}
&&\Prob \bigl(\bar\eta_T(y)\ge\ell,\bar\eta_T
\bigl(y^\prime \bigr)\ge\ell|\bar \eta_0=\ell
\delta_x \bigr)\\
&&\qquad \ge \Prob \bigl(\bar\eta_T(y)\ge\ell,
\bar \eta_T \bigl(y^\prime \bigr)\ge\ell,N_T^x
\le n|\bar\eta_0=\ell\delta_x \bigr)
\\
&&\qquad =\Prob \bigl(\widetilde\eta_T(y)\ge\ell,\widetilde
\eta_T \bigl(y^\prime \bigr)\ge\ell,N_T^x
\le n|\widetilde\eta_0=\ell\delta_x \bigr) >1-2
\varepsilon.
\end{eqnarray*}
The percolation construction of step (1) can be repeated by using
i.i.d. copies of $\{\bar\eta_t\}_{t\in[0,T]}$
instead of $\{\widetilde\eta_{(i,j)}\}_{i\in\Z,j\in\N}$. Call $\{
\bar\xi_t\}_{t\ge0}$ the corresponding
process constructed from these copies as $\{\xi_t\}_{t\ge0}$ was
constructed from $\{\widetilde\eta_{(i,j)}\}_{i\in\Z,j\in\N}$.
%
% % In the cone $\{(i,j) \dvtx j \ge|i|\}$ the existence of an
%infinite cluster containing the origin in the
% % oriented percolation is equivalent to the existence of an infinite
%cluster containing the origin in a nonoriented
% % percolation obtained from the oriented one by removing the
%orientation.
% Let us choose $\varepsilon$ such that the one-dependent percolation
%dominates a supercritical
% independent (oriented) Bernoulli percolation.
% %The sites \{i \dvtx(i,j) connected to (0,0)\} (depending on $j$)
%is a discrete-time contact process on $\mathcal{Z}$
% It is well-known that the infinite Bernoulli percolation cluster
% in the cone $\{(i,j) \dvtx j \ge|i|\}$
% contains infinitely many sites of type $(0,n)$ a.s. Hence, by
%coupling, there is a positive probability that the one-dependent
% infinite percolation cluster contains infinitely many sites of type
%$(0,n)$ as well.
As in step (1), by choosing $\varepsilon$ sufficiently small, we have
that $\bar\xi_{nT}(x_0) \ge\ell$ for infinitely many
$n \in\N$.

Let $H$ be the number of paths in ${{\mathbb Z}^d}$ of length $\bar n$,
containing
the origin: $H$ is an upper bound
for the number of such paths in ${{\mathcal C}_\infty}$ or in
$\mathcal A$. It is easy
to show that $\bar\xi_t(x)\le H\bar n$
for all $t$ and $x$. Thus if we take $k\ge H\bar n$, then $\widetilde
\eta^k_t$ stochastically dominates $\bar\xi_t$.
The supercriticality of the percolation on $\Z\times\vec\N$
associated to $\bar\xi$ implies that $\{\widetilde\eta^k_t\}_{t\ge0}$
survives locally. The inequality $\lambda> \lambda_s^k(\mathcal
{C}_\infty,\mu)$
follows since $\{\eta^k_t\}_{t\ge0}$ stochastically dominates $\{
\widetilde\eta^k_t\}_{t\ge0}$.
This implies that, for every sufficiently large $k$, $\lambda_s(\Z
^d,\mu) \le
\lambda_s^k(\Z^d,\mu) \le\lambda_s^k(\mathcal{C}_\infty,\mu)
<\lambda$ (see Figure~\ref{fig:diagram}), and the proof is complete.
% \textcolor{red}{quando costruiamo $\bar\eta_t$, $n_0$ deve essere
%sufficientemente grande da girare nel box
% e arrivare fino a $y$ e la lunghezza sufficiente \`e uniformemente
%limitata rispetto ad $x$. Inoltre i percorsi da un seed
% ai suoi due vicini possono avere vertici in comune, quindi u vertice
%pu{\chr"C3}{\chr"B2} essere ripetuto fino a due volte nel percorso che
% passa per tutti
% i seed. In ogni caso il numero di percorsi in $\mathcal{A}$ di
%lunghezza $n_0$ che passano per un punto fissato {\chr"C3}{\chr"A8}
%superiormente limitato
% dal numero di percorsi in $\Zd$ di lunghezza $n_0$ che passano per
%quello stesso punto.}
\end{pf*}

% \textcolor{red}{Mettere una figura per illustrare il processo
%rinormalizzato?}
We discuss here an interesting result on oriented percolation which is
used in the proofs of Theorem~\ref{thm:approxs} and
\cite{cf:BZ3}, Theorem~5.1.

%le3.7 #&#
\begin{lem}\label{lem:recurrence}
Consider a supercritical Bernoulli oriented percolation in $\Z\times
\vec{\N}$:
almost every infinite cluster contains
an infinite number of vertices of type $(0,n)$. The same holds for a
supercritical Bernoulli oriented percolation in $\N\times\vec{\N}$.
\end{lem}

\begin{pf}
Let us begin with the percolation in $\Z\times\vec{\N}$. By $(i,j)
\to(i^\prime,j^\prime)$ we mean that there is an open path in the percolation
from $(i,j)$ to $(i^\prime,j^\prime)$, while by $(i,j) \to\infty$
we mean that there is an infinite open path
from $(i,j)$.
By using the translation invariance of the
percolation law, the results of \cite{cf:Durrett84}, Section~3 [in
particular equations~(7) and (11)] imply that a.s.~if
$(i,j) \to\infty$, then for all $i^\prime\in\Z$ there exists
$j^\prime\ge j$ such that $(i,j) \to(i^\prime, j^\prime)$.
This implies that a.s., with respect to the percolation measure, every
vertex satisfying $(i,j) \to\infty$
has the \textsl{covering property}; that is,
if we project on the first coordinate (i.e.,~on $\Z$) the vertices in
the cluster ``branching'' from it, we obtain the whole set $\Z$.
Let $J:=\min\{j^\prime\dvtx(i,j^\prime) \to\infty\mbox{ for
some } i \in\Z\}$ be the bottom level of the infinite cluster;
for all $j \ge J$
there exists $i \in\Z$ such that $(i,j) \to\infty$.
Consider the set of infinite clusters which contain just a finite
number of vertices of type $(0,n)$;
denote by $(0,N)$ the ``highest''
of such vertices ($N$ depending on the cluster), then there exists $i$
(depending on the cluster)
such that $(i,N+1)\to\infty$. This implies that $(i,N+1)$ is in the
infinite cluster a.s.
By the covering property above, the probability that there are no paths
from $(i,N+1)$ to $(0,j^\prime)$
(for some $j^\prime\ge N+1$) is $0$. Thus the set of infinite clusters
containing a finite number
of vertices of type $(0,n)$ has probability $0$.

In the case of the oriented Bernoulli percolation on $\N\times\vec
{\N}$, we proceed analogously.
Observe that this percolation can be obtained from the oriented
Bernoulli percolation on $\Z\times\vec{\N}$ by
deleting all edges outside $\N\times\vec{\N}$; this defines a
coupling between these percolation processes.
We use here $(i,j) \rightsquigarrow(i^\prime, j^\prime)$ for an open
path in the oriented percolation in $\N\times\vec{\N}$ and again
$(i,j) \to(i^\prime, j^\prime)$ for a path in the oriented
percolation in $\Z\times\vec{\N}$ (clearly the existence of the
first one implies the existence of the last one).
Since the infinite cluster in $\Z\times\vec{\N}$ is unique a.s.,
then the infinite cluster $\N\times\vec{\N}$ is a.s.~a subset
of the previous one.
In the supercritical case, for all $j \ge\min\{j^\prime\dvtx
(i,j^\prime) \rightsquigarrow\infty\mbox{ for some } i \in\N\}$,
there exists $i\in\N$ such that $(i,j) \rightsquigarrow\infty$;
thus $(i,j) \to\infty$.
We proved before that a.s. $(i,j) \to(0,j^\prime)$ for some
$j^\prime\ge j$. Let us take the smaller of such
$j^\prime$s, say $j^\prime_0$. Hence $(i,j) \to(0,j^\prime_0)$ and
the connecting path is entirely contained in $\N\times\vec{\N}$, thus
$(i,j) \rightsquigarrow(0,j^\prime_0)$. The conclusion follows as in
the previous case.
\end{pf}

\begin{pf*}{Proof of Theorem~\ref{thm:approxw}}
If follows easily from Theorem~\ref{thm:approxs}, the hypothesis
$\lambda_s({{\mathbb Z}^d},\mu)=\lambda_w({{\mathbb Z}^d},\mu)$
and the diagram shown in Figure~\ref{fig:diagram}.
\end{pf*}

\begin{pf*}{Proof of Theorem~\ref{thm:RBRW}}
Let $\varepsilon>0$ such that $c(0)-\varepsilon>\lambda_s({{\mathbb
Z}^d},\mu)$.
By the assumptions on $c$, there exists $\delta>0$ such that
$c(z)>c(0)-\varepsilon$ for all $z\in[0,\delta]$.
From Theorem~\ref{thm:approxs} we know that there exists $k_0$ such
that the $k$-type contact
process $({{\mathcal C}_\infty},\mu)$ associated with $\lambda:=c(0)-\varepsilon$
survives locally, for all $k\ge k_0$.
Moreover, there exists $N_0$ such that $\delta N>k_0-1$ for all $N\ge N_0$.
Since $c_N(i)\ge\lambda\ident_{[0,k-1]}(i)$ for all $i\in\N$ (see
Figure~\ref{fig:c}), by coupling we have local survival
for the RBRWs $({{\mathcal C}_\infty},\mu,c_N)$ for all $N\ge N_0$.
\end{pf*}

%f6 #&#
\begin{figure}

\includegraphics{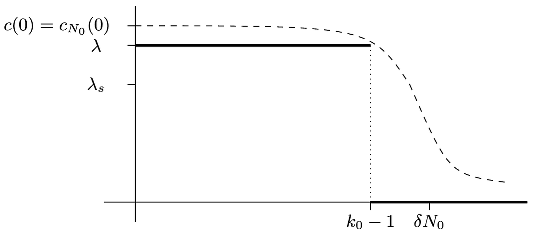}

\caption{Comparison between $c_{N_0}$ (dashed) and the $k_0$-type
contact process (thick).}
\label{fig:c}
\end{figure}

\begin{pf*}{Proof of Corollary~\ref{cor:BLZ}}
(1) It suffices to note that the total number of individuals is
dominated by the total number of
particles in a continuous-time branching process with breeding
parameter $\alpha+\beta$; for the details, see \cite{cf:BLZ},
Theorem~1(1).

(2) Note that $\mu$ is translation invariant, hence quasi-transitive.
The claim follows from Theorem~\ref{thm:RBRW} since $\lambda
_s({{\mathbb Z}^d},\mu)=(\alpha+\beta)^{-1}$ and $c(0)=1$.
\end{pf*}

%re3.8 #&#
\begin{rem}\label{rem:liggett}
In \cite{cf:LiggSpitz} the process is constructed by means
of a semigroup of operators on $\operatorname{Lip}({\mathcal{W}})$ (the space of
Lipschitz functions on the configuration space~$\mathcal{W}$).
In \cite{cf:BPZ} this technique is applied to the construction of the
restrained BRWs where $(\mu(x,y))_{x,y \in X}$ is a stochastic matrix
adapted to a graph with bounded geometry. Our definition of $\mu$ is
more general.
The only difference between the construction needed here and those in
\cite{cf:LiggSpitz,cf:BPZ} consists in the choice of
the configuration space $\mathcal{W}$ and its norm; we refer to~\cite
{cf:LiggSpitz} for the notation and details.
As in \cite{cf:LiggSpitz,cf:BPZ} we consider the space $\mathcal
{W}:=\{\eta\in\mathbb{N}^X\dvtx\sum_{x \in X} \eta(x) \alpha
(x)<+\infty\}$
where its metric is defined by $\|\eta-\bar\eta\|:=
\sum_{x \in X}|\eta(x)-\bar\eta(x)|\alpha(x)$. Our choice of the
positive function $\alpha\dvtx X \to(0,+\infty)$
is made in such a way that $\sum_{y \in X} \mu(x,y) \alpha(y) \le
\widetilde K \alpha(x)$ for all $x \in X$
(and some fixed $\widetilde K >0$). There are many ways to do this:
a possible choice is $\alpha(x):=\sum_{n=0}^\infty\widetilde K^{-n}
\sum_{y \in X}\mu^{(n)}(x,y) b(y)$
where $b\dvtx X \to(0,+\infty)$
is a fixed positive, bounded function and $\widetilde K>\sup_{x \in X}
\limsup_{n \to\infty} \sqrt[n]{\sum_{y \in X}\mu^{(n)}(x,y)}$
[take, e.g., $\widetilde K>\sup_{x \in X} \sum_{y \in X}\mu
(x,y)$]. Once $\alpha$ is chosen,
the rest of the construction of the process is carried on as
in \cite{cf:LiggSpitz,cf:BPZ}. In particular the system of
differential equations satisfied by $\{\E(\eta_t(y)|\eta_0=\delta
_x)\}_{x \in X}$
can be explicitly derived from \cite{cf:LiggSpitz}, Lemmas~2.12 and~2.16(e).
% by setting $f=e_x$ the evaluation map at $x$
% defined by $e_x(\eta):=\eta(x)$ for all $x \in X$ and $\eta\in
%\mathcal{E}$.
\end{rem}

\section*{Acknowledgment}
The authors are grateful to the anonymous referee for carefully
reading the manuscript and for useful suggestions.

% imsref loaded by akundreckaite, 2014-06-17 13:23:05
%

%\begin{appendix}
%\section{}
%\end{appendix}

%\section*{Acknowledgments}

%\begin{supplement}[id=suppA]
%\slink[doi]{10.1214/00-AAPXXXXSUPP} %[doi,text={...}] - jei reikia
%suskaldyti doi
%\sfilename{aapXXXX\_supp.pdf}
%\end{supplement}

%\begin{thebibliography}{99}
%\bibitem[\protect\citeauthoryear{}{}]{r1}
%\bibitem{r1}
%\end{thebibliography}

\printaddresses

\begin{thebibliography}{34}
% pybtex-1.11. Style name=ims, version=2.8, label_style=nolabel,
%sorting_style=complex, cfg=None, language=None.

%b1 ###
%b1 #&#
\bibitem{cf:Barlow}
%
\begin{barticle}[mr]
\bauthor{\bsnm{Barlow},~\bfnm{Martin~T.}\binits{M.~T.}}
(\byear{2004}).
\btitle{Random walks on supercritical percolation clusters}.
\bjournal{Ann. Probab.}
\bvolume{32}
\bpages{3024--3084}.
\bid{doi={10.1214/009117904000000748}, issn={0091-1798}, mr={2094438}}
\end{barticle}
%
\bptok{imsref}%
% NOT OUTPUTED:
% issn = 0091-1798
% url = http://dx.doi.org/10.1214/009117904000000748
% number = 4
% coden = APBYAE
% fjournal = The Annals of Probability
\endbibitem

%b2 ###
%b2 #&#
\bibitem{cf:BarPerSou}
%
\begin{barticle}[mr]
\bauthor{\bsnm{Barlow},~\bfnm{Martin~T.}\binits{M.~T.}},
\bauthor{\bsnm{Peres},~\bfnm{Yuval}\binits{Y.}} \AND
\bauthor{\bsnm{Sousi},~\bfnm{Perla}\binits{P.}}
(\byear{2012}).
\btitle{Collisions of random walks}.
\bjournal{Ann. Inst. Henri Poincar\'e Probab. Stat.}
\bvolume{48}
\bpages{922--946}.
\bid{doi={10.1214/12-AIHP481}, issn={0246-0203}, mr={3052399}}
\end{barticle}
%
\bptok{imsref}%
% NOT OUTPUTED:
% issn = 0246-0203
% url = http://dx.doi.org/10.1214/12-AIHP481
% number = 4
% fjournal = Annales de l'Institut Henri Poincar\'e Probabilit\'es et
%Statistiques
\endbibitem

%b3 ###
%b3 #&#
\bibitem{cf:BBZ}
%
\begin{barticle}[mr]
\bauthor{\bsnm{Belhadji},~\bfnm{Lamia}\binits{L.}},
\bauthor{\bsnm{Bertacchi},~\bfnm{Daniela}\binits{D.}} \AND
\bauthor{\bsnm{Zucca},~\bfnm{Fabio}\binits{F.}}
(\byear{2010}).
\btitle{A self-regulating and patch subdivided population}.
\bjournal{Adv. in Appl. Probab.}
\bvolume{42}
\bpages{899--912}.
\bid{doi={10.1239/aap/1282924068}, issn={0001-8678}, mr={2779564}}
\end{barticle}
%
\bptok{imsref}%
% NOT OUTPUTED:
% issn = 0001-8678
% url = http://dx.doi.org/10.1239/aap/1282924068
% number = 3
% coden = AAPBBD
% fjournal = Advances in Applied Probability
\endbibitem

%b4 ###
%b4 #&#
\bibitem{cf:BL}
%
\begin{barticle}[mr]
\bauthor{\bsnm{Belhadji},~\bfnm{L.}\binits{L.}} \AND
\bauthor{\bsnm{Lanchier},~\bfnm{N.}\binits{N.}}
(\byear{2006}).
\btitle{Individual versus cluster recoveries within a spatially
structured population}.
\bjournal{Ann. Appl. Probab.}
\bvolume{16}
\bpages{403--422}.
\bid{doi={10.1214/105051605000000764}, issn={1050-5164}, mr={2209347}}
\end{barticle}
%
\bptok{imsref}%
% NOT OUTPUTED:
% issn = 1050-5164
% url = http://dx.doi.org/10.1214/105051605000000764
% number = 1
% fjournal = The Annals of Applied Probability
\endbibitem

%b5 ###
%b5 #&#
\bibitem{cf:BerBor}
%
\begin{barticle}[mr]
\bauthor{\bsnm{Bertacchi},~\bfnm{Daniela}\binits{D.}} \AND
\bauthor{\bsnm{Borrello},~\bfnm{Davide}\binits{D.}}
(\byear{2011}).
\btitle{The small world effect on the coalescing time of random walks}.
\bjournal{Stochastic Process. Appl.}
\bvolume{121}
\bpages{925--956}.
\bid{doi={10.1016/j.spa.2011.01.003}, issn={0304-4149}, mr={2775102}}
\end{barticle}
%
\bptok{imsref}%
% NOT OUTPUTED:
% issn = 0304-4149
% url = http://dx.doi.org/10.1016/j.spa.2011.01.003
% number = 5
% coden = STOPB7
% fjournal = Stochastic Processes and their Applications
\endbibitem

%b6 ###
%b6 #&#
\bibitem{cf:BLZ}
%
\begin{barticle}[mr]
\bauthor{\bsnm{Bertacchi},~\bfnm{D.}\binits{D.}},
\bauthor{\bsnm{Lanchier},~\bfnm{N.}\binits{N.}} \AND
\bauthor{\bsnm{Zucca},~\bfnm{F.}\binits{F.}}
(\byear{2011}).
\btitle{Contact and voter processes on the infinite percolation
cluster as models of host-symbiont interactions}.
\bjournal{Ann. Appl. Probab.}
\bvolume{21}
\bpages{1215--1252}.
\bid{doi={10.1214/10-AAP734}, issn={1050-5164}, mr={2857447}}
\end{barticle}
%
\bptok{imsref}%
% NOT OUTPUTED:
% issn = 1050-5164
% url = http://dx.doi.org/10.1214/10-AAP734
% number = 4
% fjournal = The Annals of Applied Probability
\endbibitem

%b7 ###
%b7 #&#
\bibitem{cf:BPZ}
%
\begin{barticle}[mr]
\bauthor{\bsnm{Bertacchi},~\bfnm{Daniela}\binits{D.}},
\bauthor{\bsnm{Posta},~\bfnm{Gustavo}\binits{G.}} \AND
\bauthor{\bsnm{Zucca},~\bfnm{Fabio}\binits{F.}}
(\byear{2007}).
\btitle{Ecological equilibrium for restrained branching random walks}.
\bjournal{Ann. Appl. Probab.}
\bvolume{17}
\bpages{1117--1137}.
\bid{doi={10.1214/105051607000000203}, issn={1050-5164}, mr={2344301}}
\end{barticle}
%
\bptok{imsref}%
% NOT OUTPUTED:
% issn = 1050-5164
% url = http://dx.doi.org/10.1214/105051607000000203
% number = 4
% fjournal = The Annals of Applied Probability
\endbibitem

%b8 ###
%b8 #&#
\bibitem{cf:BZ03}
%
\begin{barticle}[mr]
\bauthor{\bsnm{Bertacchi},~\bfnm{Daniela}\binits{D.}} \AND
\bauthor{\bsnm{Zucca},~\bfnm{Fabio}\binits{F.}}
(\byear{2003}).
\btitle{Uniform asymptotic estimates of transition probabilities on combs}.
\bjournal{J. Aust. Math. Soc.}
\bvolume{75}
\bpages{325--353}.
\bid{doi={10.1017/S1446788700008144}, issn={1446-7887}, mr={2015321}}
\end{barticle}
%
\bptok{imsref}%
% NOT OUTPUTED:
% issn = 1446-7887
% url = http://dx.doi.org/10.1017/S1446788700008144
% number = 3
% fjournal = Journal of the Australian Mathematical Society
\endbibitem

%b9 ###
%b9 #&#
\bibitem{cf:BZ}
%
\begin{barticle}[mr]
\bauthor{\bsnm{Bertacchi},~\bfnm{Daniela}\binits{D.}} \AND
\bauthor{\bsnm{Zucca},~\bfnm{Fabio}\binits{F.}}
(\byear{2008}).
\btitle{Critical behaviors and critical values of branching random
walks on multigraphs}.
\bjournal{J. Appl. Probab.}
\bvolume{45}
\bpages{481--497}.
\bid{doi={10.1239/jap/1214950362}, issn={0021-9002}, mr={2426846}}
\end{barticle}
%
\bptok{imsref}%
% NOT OUTPUTED:
% issn = 0021-9002
% url = http://dx.doi.org/10.1239/jap/1214950362
% number = 2
% coden = JPRBAM
% fjournal = Journal of Applied Probability
\endbibitem

%b10 ###
%b10 #&#
\bibitem{cf:BZ2}
%
\begin{barticle}[mr]
\bauthor{\bsnm{Bertacchi},~\bfnm{Daniela}\binits{D.}} \AND
\bauthor{\bsnm{Zucca},~\bfnm{Fabio}\binits{F.}}
(\byear{2009}).
\btitle{Characterization of critical values of branching random walks
on weighted graphs through infinite-type branching processes}.
\bjournal{J. Stat. Phys.}
\bvolume{134}
\bpages{53--65}.
\bid{doi={10.1007/s10955-008-9653-5}, issn={0022-4715}, mr={2489494}}
\end{barticle}
%
\bptok{imsref}%
% NOT OUTPUTED:
% issn = 0022-4715
% url = http://dx.doi.org/10.1007/s10955-008-9653-5
% number = 1
% fjournal = Journal of Statistical Physics
\endbibitem

%b11 ###
%b11 #&#
\bibitem{cf:BZ3}
%
\begin{barticle}[mr]
\bauthor{\bsnm{Bertacchi},~\bfnm{Daniela}\binits{D.}} \AND
\bauthor{\bsnm{Zucca},~\bfnm{Fabio}\binits{F.}}
(\byear{2009}).
\btitle{Approximating critical parameters of branching random walks}.
\bjournal{J. Appl. Probab.}
\bvolume{46}
\bpages{463--478}.
\bid{doi={10.1239/jap/1245676100}, issn={0021-9002}, mr={2535826}}
\end{barticle}
%
\bptok{imsref}%
% NOT OUTPUTED:
% issn = 0021-9002
% url = http://dx.doi.org/10.1239/jap/1245676100
% number = 2
% coden = JPRBAM
% fjournal = Journal of Applied Probability
\endbibitem

%b12 ###
%b12 #&#
\bibitem{cf:BZsl}
%
\begin{barticle}[mr]
\bauthor{\bsnm{Bertacchi},~\bfnm{Daniela}\binits{D.}} \AND
\bauthor{\bsnm{Zucca},~\bfnm{Fabio}\binits{F.}}
(\byear{2014}).
\btitle{Strong local survival of branching random walks is not monotone}.
\bjournal{Adv. in Appl. Probab.}
\bvolume{46}
\bpages{400--421}.
\bid{doi={10.1239/aap/1401369700}, issn={0001-8678}, mr={3215539}}
\end{barticle}
%
\bptok{imsref}%
% NOT OUTPUTED:
% issn = 0001-8678
% url = http://dx.doi.org/10.1239/aap/1401369700
% number = 2
% fjournal = Advances in Applied Probability
\endbibitem

%b13 ###
%b13 #&#
\bibitem{cf:CMP98}
%
\begin{barticle}[mr]
\bauthor{\bsnm{Comets},~\bfnm{F.}\binits{F.}},
\bauthor{\bsnm{Menshikov},~\bfnm{M.~V.}\binits{M.~V.}} \AND
\bauthor{\bsnm{Popov},~\bfnm{S.~Y.}\binits{S.~Y.}}
(\byear{1998}).
\btitle{One-dimensional branching random walk in a random environment:
A classification}.
\bjournal{Markov Process. Related Fields}
\bvolume{4}
\bpages{465--477}.
%\bnote{I Brazilian School in Probability (Rio de Janeiro, 1997)}.
\bid{issn={1024-2953}, mr={1677053}}
\end{barticle}
%
\bptok{imsref}%
% NOT OUTPUTED:
% issn = 1024-2953
% number = 4
% fjournal = Markov Processes and Related Fields
\endbibitem

%b14 ###
%b14 #&#
\bibitem{cf:CM2004}
%
\begin{barticle}[mr]
\bauthor{\bsnm{Couronn{\'e}},~\bfnm{Olivier}\binits{O.}} \AND
\bauthor{\bsnm{Messikh},~\bfnm{Reda~J{\"u}rg}\binits{R.~J.}}
(\byear{2004}).
\btitle{Surface order large deviations for 2{D} {FK}-percolation and
{P}otts models}.
\bjournal{Stochastic Process. Appl.}
\bvolume{113}
\bpages{81--99}.
\bid{doi={10.1016/j.spa.2004.03.010}, issn={0304-4149}, mr={2078538}}
\end{barticle}
%
\bptok{imsref}%
% NOT OUTPUTED:
% issn = 0304-4149
% url = http://dx.doi.org/10.1016/j.spa.2004.03.010
% number = 1
% coden = STOPB7
% fjournal = Stochastic Processes and their Applications
\endbibitem

%b15 ###
%b15 #&#
\bibitem{cf:DHMP99}
%
\begin{barticle}[mr]
\bauthor{\bsnm{den Hollander},~\bfnm{F.}\binits{F.}},
\bauthor{\bsnm{Menshikov},~\bfnm{M.~V.}\binits{M.~V.}} \AND
\bauthor{\bsnm{Popov},~\bfnm{S.~Y.}\binits{S.~Y.}}
(\byear{1999}).
\btitle{A note on transience versus recurrence for a branching random
walk in random environment}.
\bjournal{J. Stat. Phys.}
\bvolume{95}
\bpages{587--614}.
\bid{doi={10.1023/A:1004539225064}, issn={0022-4715}, mr={1700867}}
\end{barticle}
%
\bptok{imsref}%
% NOT OUTPUTED:
% issn = 0022-4715
% url = http://dx.doi.org/10.1023/A:1004539225064
% number = 3-4
% coden = JSTPSB
% fjournal = Journal of Statistical Physics
\endbibitem

%b16 ###
%b16 #&#
\bibitem{cf:Durrett84}
%
\begin{barticle}[mr]
\bauthor{\bsnm{Durrett},~\bfnm{Richard}\binits{R.}}
(\byear{1984}).
\btitle{Oriented percolation in two dimensions}.
\bjournal{Ann. Probab.}
\bvolume{12}
\bpages{999--1040}.
\bid{issn={0091-1798}, mr={0757768}}
\end{barticle}
%
\bptok{imsref}%
% NOT OUTPUTED:
% issn = 0091-1798
% url =
%%%http://links.jstor.org/sici?sici=0091-1798(198411)12:4<999:OPITD>2.0.CO;2-4&origin=MSN
% number = 4
% coden = APBYAE
% fjournal = The Annals of Probability
\endbibitem

%b17 ###
%b17 #&#
\bibitem{cf:DurrettJung}
%
\begin{barticle}[mr]
\bauthor{\bsnm{Durrett},~\bfnm{Rick}\binits{R.}} \AND
\bauthor{\bsnm{Jung},~\bfnm{Paul}\binits{P.}}
(\byear{2007}).
\btitle{Two phase transitions for the contact process on small worlds}.
\bjournal{Stochastic Process. Appl.}
\bvolume{117}
\bpages{1910--1927}.
\bid{doi={10.1016/j.spa.2007.03.003}, issn={0304-4149}, mr={2437735}}
\end{barticle}
%
\bptok{imsref}%
% NOT OUTPUTED:
% issn = 0304-4149
% url = http://dx.doi.org/10.1016/j.spa.2007.03.003
% number = 12
% coden = STOPB7
% fjournal = Stochastic Processes and their Applications
\endbibitem

%b18 ###
%b18 #&#
\bibitem{cf:GMPV09}
%
\begin{barticle}[mr]
\bauthor{\bsnm{Gantert},~\bfnm{Nina}\binits{N.}},
\bauthor{\bsnm{M{\"u}ller},~\bfnm{Sebastian}\binits{S.}},
\bauthor{\bsnm{Popov},~\bfnm{Serguei}\binits{S.}} \AND
\bauthor{\bsnm{Vachkovskaia},~\bfnm{Marina}\binits{M.}}
(\byear{2010}).
\btitle{Survival of branching random walks in random environment}.
\bjournal{J. Theoret. Probab.}
\bvolume{23}
\bpages{1002--1014}.
\bid{doi={10.1007/s10959-009-0227-5}, issn={0894-9840}, mr={2735734}}
\end{barticle}
%
\bptok{imsref}%
% NOT OUTPUTED:
% issn = 0894-9840
% url = http://dx.doi.org/10.1007/s10959-009-0227-5
% number = 4
% coden = JTPREO
% fjournal = Journal of Theoretical Probability
\endbibitem

%b19 ###
%b19 #&#
\bibitem{cf:Grimmett}
%
\begin{bbook}[mr]
\bauthor{\bsnm{Grimmett},~\bfnm{Geoffrey}\binits{G.}}
(\byear{1999}).
\btitle{Percolation},
\bedition{2nd} ed.
\bseries{Grundlehren der Mathematischen Wissenschaften}
\bvolume{321}.
\bpublisher{Springer},
\blocation{Berlin}.
\bid{doi={10.1007/978-3-662-03981-6}, mr={1707339}}
\end{bbook}
%
\bptok{imsref}%
% NOT OUTPUTED:
% isbn = 3-540-64902-6
% url = http://dx.doi.org/10.1007/978-3-662-03981-6
% fpage = xiv+444
\endbibitem

%b20 ###
%b20 #&#
\bibitem{cf:HuLalley}
%
\begin{barticle}[mr]
\bauthor{\bsnm{Hueter},~\bfnm{Irene}\binits{I.}} \AND
\bauthor{\bsnm{Lalley},~\bfnm{Steven~P.}\binits{S.~P.}}
(\byear{2000}).
\btitle{Anisotropic branching random walks on homogeneous trees}.
\bjournal{Probab. Theory Related Fields}
\bvolume{116}
\bpages{57--88}.
\bid{doi={10.1007/PL00008723}, issn={0178-8051}, mr={1736590}}
\end{barticle}
%
\bptok{imsref}%
% NOT OUTPUTED:
% issn = 0178-8051
% url = http://dx.doi.org/10.1007/PL00008723
% number = 1
% coden = PTRFEU
% fjournal = Probability Theory and Related Fields
\endbibitem

%b21 ###
%b21 #&#
\bibitem{cf:Ligg1}
%
\begin{barticle}[mr]
\bauthor{\bsnm{Liggett},~\bfnm{Thomas~M.}\binits{T.~M.}}
(\byear{1996}).
\btitle{Branching random walks and contact processes on homogeneous trees}.
\bjournal{Probab. Theory Related Fields}
\bvolume{106}
\bpages{495--519}.
\bid{doi={10.1007/s004400050073}, issn={0178-8051}, mr={1421990}}
\end{barticle}
%
\bptok{imsref}%
% NOT OUTPUTED:
% issn = 0178-8051
% url = http://dx.doi.org/10.1007/s004400050073
% number = 4
% coden = PTRFEU
% fjournal = Probability Theory and Related Fields
\endbibitem

%b22 ###
%b22 #&#
\bibitem{cf:Ligg2}
%
\begin{bincollection}[mr]
\bauthor{\bsnm{Liggett},~\bfnm{Thomas~M.}\binits{T.~M.}}
(\byear{1999}).
\btitle{Branching random walks on finite trees}.
In \bbooktitle{Perplexing Problems in Probability}.
\bseries{Progress in Probability}
\bvolume{44}
\bpages{315--330}.
\bpublisher{Birkh\"auser},
\blocation{Boston, MA}.
\bid{mr={1703138}}
\end{bincollection}
%
\bptok{imsref}%
\endbibitem

%b23 ###
%b23 #&#
\bibitem{cf:LiggSpitz}
%
\begin{barticle}[mr]
\bauthor{\bsnm{Liggett},~\bfnm{Thomas~M.}\binits{T.~M.}} \AND
\bauthor{\bsnm{Spitzer},~\bfnm{Frank}\binits{F.}}
(\byear{1981}).
\btitle{Ergodic theorems for coupled random walks and other systems
with locally interacting components}.
\bjournal{Z. Wahrsch. Verw. Gebiete}
\bvolume{56}
\bpages{443--468}.
\bid{doi={10.1007/BF00531427}, issn={0044-3719}, mr={0621659}}
\bptnote{check year}%
\end{barticle}
%
\bptok{imsref}%
% NOT OUTPUTED:
% issn = 0044-3719
% url = http://dx.doi.org/10.1007/BF00531427
% number = 4
% fjournal = Zeitschrift f\"ur Wahrscheinlichkeitstheorie und Verwandte
%Gebiete
\endbibitem

%b24 ###
%b24 #&#
\bibitem{cf:MP00}
%
\begin{barticle}[mr]
\bauthor{\bsnm{Machado},~\bfnm{F.~P.}\binits{F.~P.}} \AND
\bauthor{\bsnm{Popov},~\bfnm{S.~Y.}\binits{S.~Y.}}
(\byear{2000}).
\btitle{One-dimensional branching random walks in a {M}arkovian random
environment}.
\bjournal{J. Appl. Probab.}
\bvolume{37}
\bpages{1157--1163}.
\bid{issn={0021-9002}, mr={1808881}}
\end{barticle}
%
\bptok{imsref}%
% NOT OUTPUTED:
% issn = 0021-9002
% number = 4
% coden = JPRBAM
% fjournal = Journal of Applied Probability
\endbibitem

%b25 ###
%b25 #&#
\bibitem{cf:MP03}
%
\begin{barticle}[mr]
\bauthor{\bsnm{Machado},~\bfnm{F.~P.}\binits{F.~P.}} \AND
\bauthor{\bsnm{Popov},~\bfnm{S.~Y.}\binits{S.~Y.}}
(\byear{2003}).
\btitle{Branching random walk in random environment on trees}.
\bjournal{Stochastic Process. Appl.}
\bvolume{106}
\bpages{95--106}.
\bid{issn={0304-4149}, mr={1983045}}
\end{barticle}
%
\bptok{imsref}%
% NOT OUTPUTED:
% issn = 0304-4149
% url =
%%%http://www.sciencedirect.com/science?_ob=GatewayURL&_origin=MR&_method=citationSearch&_piikey=s0304414903000395&_version=1&md5=f527005ab84d253b400039675a91af1b
% number = 1
% coden = STOPB7
% fjournal = Stochastic Processes and their Applications
\endbibitem

%b26 ###
%b26 #&#
\bibitem{cf:MadrasSchi}
%
\begin{barticle}[mr]
\bauthor{\bsnm{Madras},~\bfnm{Neal}\binits{N.}} \AND
\bauthor{\bsnm{Schinazi},~\bfnm{Rinaldo}\binits{R.}}
(\byear{1992}).
\btitle{Branching random walks on trees}.
\bjournal{Stochastic Process. Appl.}
\bvolume{42}
\bpages{255--267}.
\bid{doi={10.1016/0304-4149(92)90038-R}, issn={0304-4149}, mr={1176500}}
\end{barticle}
%
\bptok{imsref}%
% NOT OUTPUTED:
% issn = 0304-4149
% url = http://dx.doi.org/10.1016/0304-4149(92)90038-R
% number = 2
% coden = STOPB7
% fjournal = Stochastic Processes and their Applications
\endbibitem

%b27 ###
%b27 #&#
\bibitem{cf:Muller14}
%
\begin{barticle}[auto:STB|2014/06/10|07:15:57]
\bauthor{\bsnm{M{\"u}ller},~\bfnm{S.}\binits{S.}}
(\byear{2014}).
\btitle{Interacting growth processes and invariant percolation}.
\bjournal{Ann. Appl. Probab.}
\bvolume{25}
\bpages{268--286}.
\bid{mr={3297773}}
\end{barticle}
%
\bptok{imsref}%
% NOT OUTPUTED:
% sortkey = Muller
% howpublished =
\endbibitem

%b28 ###
%b28 #&#
\bibitem{cf:M08}
%
\begin{barticle}[mr]
\bauthor{\bsnm{M{\"u}ller},~\bfnm{Sebastian}\binits{S.}}
(\byear{2008}).
\btitle{A criterion for transience of multidimensional branching
random walk in random environment}.
\bjournal{Electron. J. Probab.}
\bvolume{13}
\bpages{1189--1202}.
\bid{doi={10.1214/EJP.v13-517}, issn={1083-6489}, mr={2430704}}
\end{barticle}
%
\bptok{imsref}%
% NOT OUTPUTED:
% issn = 1083-6489
% url = http://dx.doi.org/10.1214/EJP.v13-517
% fjournal = Electronic Journal of Probability
\endbibitem

%b29 ###
%b29 #&#
\bibitem{cf:Pem}
%
\begin{barticle}[mr]
\bauthor{\bsnm{Pemantle},~\bfnm{Robin}\binits{R.}}
(\byear{1992}).
\btitle{The contact process on trees}.
\bjournal{Ann. Probab.}
\bvolume{20}
\bpages{2089--2116}.
\bid{issn={0091-1798}, mr={1188054}}
\end{barticle}
%
\bptok{imsref}%
% NOT OUTPUTED:
% issn = 0091-1798
% url =
%%%http://links.jstor.org/sici?sici=0091-1798(199210)20:4<2089:TCPOT>2.0.CO;2-1&origin=MSN
% number = 4
% coden = APBYAE
% fjournal = The Annals of Probability
\endbibitem

%b30 ###
%b30 #&#
\bibitem{cf:PemStac1}
%
\begin{barticle}[mr]
\bauthor{\bsnm{Pemantle},~\bfnm{Robin}\binits{R.}} \AND
\bauthor{\bsnm{Stacey},~\bfnm{Alan~M.}\binits{A.~M.}}
(\byear{2001}).
\btitle{The branching random walk and contact process on
{G}alton--{W}atson and nonhomogeneous trees}.
\bjournal{Ann. Probab.}
\bvolume{29}
\bpages{1563--1590}.
\bid{doi={10.1214/aop/1015345762}, issn={0091-1798}, mr={1880232}}
\end{barticle}
%
\bptok{imsref}%
% NOT OUTPUTED:
% issn = 0091-1798
% url = http://dx.doi.org/10.1214/aop/1015345762
% number = 4
% coden = APBYAE
% fjournal = The Annals of Probability
\endbibitem

%b31 ###
%b31 #&#
\bibitem{cf:Pisz}
%
\begin{barticle}[mr]
\bauthor{\bsnm{Pisztora},~\bfnm{Agoston}\binits{A.}}
(\byear{1996}).
\btitle{Surface order large deviations for {I}sing, {P}otts and
percolation models}.
\bjournal{Probab. Theory Related Fields}
\bvolume{104}
\bpages{427--466}.
\bid{doi={10.1007/BF01198161}, issn={0178-8051}, mr={1384040}}
\end{barticle}
%
\bptok{imsref}%
% NOT OUTPUTED:
% issn = 0178-8051
% url = http://dx.doi.org/10.1007/BF01198161
% number = 4
% coden = PTRFEU
% fjournal = Probability Theory and Related Fields
\endbibitem

%b32 ###
%b32 #&#
\bibitem{cf:Sen}
%
\begin{bbook}[mr]
\bauthor{\bsnm{Seneta},~\bfnm{E.}\binits{E.}}
(\byear{2006}).
\btitle{Non-negative Matrices and {M}arkov Chains}.
%\bseries{Springer Series in Statistics}.
\bpublisher{Springer},
\blocation{New York}.
%\bnote{Revised reprint of the second (1981) edition [Springer-Verlag,
%New York; MR0719544]}.
\bid{mr={2209438}}
\end{bbook}
%
\bptok{imsref}%
% NOT OUTPUTED:
% isbn = 978-0387-29765-1; 0-387-29765-0
% fpage = xvi+287
\endbibitem

%b33 ###
%b33 #&#
\bibitem{cf:Stacey03}
%
\begin{barticle}[mr]
\bauthor{\bsnm{Stacey},~\bfnm{Alan}\binits{A.}}
(\byear{2003}).
\btitle{Branching random walks on quasi-transitive graphs}.
\bjournal{Combin. Probab. Comput.}
\bvolume{12}
\bpages{345--358}.
%\bnote{Combinatorics, probability and computing (Oberwolfach, 2001)}.
\bid{doi={10.1017/S0963548302005588}, issn={0963-5483}, mr={1988981}}
\end{barticle}
%
\bptok{imsref}%
% NOT OUTPUTED:
% issn = 0963-5483
% url = http://dx.doi.org/10.1017/S0963548302005588
% number = 3
% fjournal = Combinatorics, Probability and Computing
\endbibitem

%b34 ###
%b34 #&#
\bibitem{cf:Z1}
%
\begin{barticle}[mr]
\bauthor{\bsnm{Zucca},~\bfnm{Fabio}\binits{F.}}
(\byear{2011}).
\btitle{Survival, extinction and approximation of discrete-time
branching random walks}.
\bjournal{J. Stat. Phys.}
\bvolume{142}
\bpages{726--753}.
\bid{doi={10.1007/s10955-011-0134-x}, issn={0022-4715}, mr={2773785}}
\end{barticle}
%
\bptok{imsref}%
% NOT OUTPUTED:
% issn = 0022-4715
% url = http://dx.doi.org/10.1007/s10955-011-0134-x
% number = 4
% fjournal = Journal of Statistical Physics
\endbibitem
\end{thebibliography}
\end{document}